\documentclass[11pt,twoside]{amsart}
\usepackage{amsmath, amsthm, amscd, amsfonts, amssymb, graphicx, color}
\usepackage[bookmarksnumbered, plainpages]{hyperref}
\usepackage{tikz}
\usetikzlibrary{arrows}
\usepackage{amssymb,tikz}
\usetikzlibrary{arrows,shapes,positioning,decorations.markings}
\usepackage[margin=1.7cm]{geometry}
\numberwithin{equation}{section}

\newtheorem{observation}{Observation}[section]
\newtheorem{lemma}[observation]{Lemma}
\newtheorem{theorem}[observation]{Theorem}
\newtheorem{proposition}[observation]{Proposition}

\newtheorem{example}[observation]{Example}
\numberwithin{equation}{section}

\newtheorem{remark}[observation]{Remark}
\newtheorem{question}[observation]{Question}

\newtheorem{problem}[observation]{Problem}

\newcommand{\cd}{\mathrm{cd}}

\newcommand{\Irr}{\mathrm{Irr}}
\def\Z#1{{\bf Z}{{(#1)}}}

\def\F#1{{\bf F}{{(#1)}}}
\def\cent#1#2{{\bf C}_{{#1}}{{(#2)}}}

\DeclareMathOperator{\diam}{{\rm diam}}

\def\diam{{\rm diam}}

\def\g\cd{{\rm g\cd}}

\def\cent#1#2{{\bf C}_{#1}(#2)}

\def\Sm0#1{{{\rm GL}}(1,#1)}

\begin{document}

\vspace{1.3 cm}

\title{An overview on the bipartite divisor graph for the set of irreducible character degrees}

\author[R. Hafezieh]{Roghayeh Hafezieh}
\address{Roghayeh Hafezieh, Department of
Mathematics, Gebze Technical University, P.O.Box 41400, Gebze, Turkey}
\email{roghayeh@gtu.edu.tr}
\author[P. Spiga]{Pablo Spiga}
\address{Pablo Spiga, Dipartimento di Matematica Pura e Applicata,
 University of Milano-Bicocca, Via Cozzi 55, 20126 Milano, Italy}
\email{pablo.spiga@unimib.it}

\thanks{{\scriptsize
\hskip -0.4 true cm MSC(2010): Primary 05C25; secondary 05C75.
$*$Corresponding author: roghayeh@gtu.edu.tr}}
\maketitle

\begin{abstract}
Let $G$ be a finite group. The bipartite divisor graph $B(G)$ for the set of irreducible complex character degrees $\cd(G)$ is the undirected graph with vertex set consisting of the prime numbers dividing some element of $\cd(G)$ and of the non-identity character degrees in $\cd(G)$, where a prime number $p$ is declared to be adjacent to a character degree $m$ if and only if $p$ divides $m$. The graph $B(G)$ is bipartite and it encodes two of the most widely studied graphs associated to the character degrees of a finite group: the prime graph and the divisor graph on the set of irreducible character degrees.

The scope of this paper is two-fold. We draw some attention to $B(G)$ by outlining the main results that have been proved so far, see for instance~\cite{H,Hregular, mus, moo4, moo5}. In this process we improve some of these results.
\end{abstract}

\vskip 0.2 true cm

\pagestyle{myheadings}
%

\bigskip

\section{Introduction}
An active line of research studies the relations between structural
properties of  groups and sets of invariants. There is a large number of examples of this and we name only four which might be considered as the genesis of this type of investigations; these four examples should also clarify our interest in this paper. Given a finite group $G$, we may associate the prime graph based on the conjugacy class sizes: the vertices are the prime numbers dividing the cardinality of some conjugacy class of $G$ and the prime numbers $p$ and $q$ are declared to be adjacent if and only if $pq$ divides the cardinality of some conjugacy class of $G$. On the set of conjugacy class sizes, we may also associate the divisor graph: the vertices are the  cardinalities of the non-central conjugacy classes of $G$ and the numbers $m$ and $n$ are declared to be adjacent if and only if they are not relatively prime. It is well known that a great deal of information on $G$ is encoded in both of these graphs and this establishes a beautiful flow of information between the algebraic structure of $G$ and the combinatorial properties of the graphs. Two entirely similar constructions can be done replacing the set of conjugacy class sizes with the set of irreducible complex character degrees. Again, the prime graph and the divisor graph on the character degrees encode interesting information about the group. Considering the ``duality'' between conjugacy classes and irreducible characters there are also some remarkable connections among all four of these graphs.

In~\cite{L}, Mark L. Lewis has generalized in a very natural and useful way these graphs. Given a subset $X\subseteq \mathbb{N}\setminus\{0,1\}$, Lewis has considered the {\em prime graph} $\Delta(X)$ and  the {\em divisor graph} $\Gamma(X)$. The vertices of $\Delta(X)$ are the prime numbers dividing some element of $X$ and two distinct prime numbers are declared to be adjacent if and only if their product divides some member of $X$. The vertex set of $\Gamma(X)$ is $X$ and two distinct elements of $X$ are declared to be adjacent if and only if they are not relatively prime. Then, Lewis has shown (for arbitrary sets $X$) some remarkable general connections between $\Delta(X)$ and $\Gamma(X)$. By taking $X$ the set of conjugacy class sizes or the set of irreducible complex characters, one recovers the graphs introduced in the previous paragraph and rediscovers some of their basic relations.

There is a gadget that can be used to study simultaneously $\Delta(X)$ and $\Gamma(X)$. Inspired by the remarkable connections between the common divisor graph $\Gamma(X)$ and the prime degree graph $\Delta(X)$ discussed in ~\cite{L} by Lewis,  Iranmanesh and Praeger~\cite{IP} introduced the notion of {\em bipartite divisor graph} $B(X)$, and proved that most of these connections follow immediately from $B(X)$. The vertex set of $B(X)$ is the disjoint union of the set of prime numbers dividing some element of $X$ and the set $X$ itself, where a prime number $p$ is declared to be adjacent to an element $x$ of $X$ if and only if $p$ divides $x$. For instance, with this new tool, Iranmanesh and Praeger (re)established the links between the number of connected components and the diameters of $\Delta(X)$ and $\Gamma(X)$ simply working with $B(X)$. They were also able to classify the graphs $\Gamma$ with $\Gamma\cong B(X)$, for some set $X$.

Before continuing our discussion, it is very important to observe that $B(X)$ brings  more information than $\Delta(X)$ and $\Gamma(X)$. In other words, the graph $B(X)$ cannot be recovered only from $\Delta(X)$ and $\Gamma(X)$. For instance, when $\Delta(X)\cong \Gamma(X)$ is the complete graph $K_3$ on three vertices, the graph $B(X)$ can be isomorphic to one of the two graphs in Figure~\ref{fig1fig1}.
\begin{figure}[!ht]
\begin{tikzpicture}
  [scale=1,auto=left,every node/.style={circle,fill=blue!20}]
\node[fill=blue] (m1) at (-3,1)  {};
\node[fill=blue] (m2) at (-4,1)  {};
\node[fill=blue] (m3) at (-5,1)  {};
\node (m4) at (-3,-1)  {};
\node (m5) at (-4,-1) {};
\node (m6) at (-5,-1) {};
\node[fill=blue] (n1) at (3,1)  {};
\node[fill=blue] (n2) at (4,1)  {};
\node[fill=blue] (n3) at (2,1)  {};
\node (n4) at (3,-1)  {};
\node (n5) at (4,-1) {};
\node (n6) at (2,-1) {};
\draw (m1) to (m4);
\draw (m1) to (m5);
\draw (m2) to (m4);
\draw (m2) to (m6);
\draw (m3) to (m6);
\draw (m3) to (m5);
\draw (n1) to (n6);
\draw (n1) to (n5);
\draw (n1) to (n4);
\draw (n3) to (n4);
\draw (n2) to (n4);
\end{tikzpicture}
\caption{Examples of two bipartite divisor graphs $B(X)$ giving rise to $\Delta(X)\cong \Gamma(X)\cong K_3$}
\label{fig1fig1}
\end{figure}
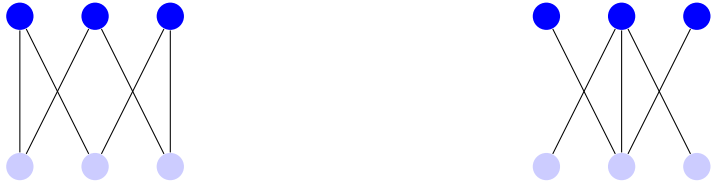 These graphs arise by taking $p,q$ and $r$ three distinct primes and by taking, for instance, $X=\{pq,pr,qr\}$ or $X=\{pqr,p,p^2\}$. (There are other isomorphism classes of $B(X)$ yielding $\Delta(X)\cong \Gamma(X)\cong K_3$, here we just presented two.)
It goes without saying that the extra information brought by  $B(X)$ asks for a finer investigation.

The first application of the bipartite divisor graph in group theory is for the set of conjugacy class sizes, that is, for a given finite group $G$, $X:=\{m\in\mathbb{N}\mid m \textrm{ is the conjugacy class size of some non-central element of }G\}$. (There is no standard notation for this graph and, in this section, we denote it by $B(Cl(G))$.) In ~\cite{BDIP}, the authors considered this graph  and studied various properties. Among other things, they proved that the diameter is at most $6$ and they classified the groups attaining the upper bound. Moreover, when the graph has no cycles, the diameter is actually at most $5$ and they classified the groups for which the graph is a path of length $5$.

The classification of Dolfi and Jabara~\cite{DJ}  of the finite groups with only two non-trivial conjugacy class sizes has spurred more interest in $B(Cl(G))$ and has proved useful in studying $B(Cl(G))$. For instance,  it follows immediately from~\cite{DJ} that there is no group $G$ with $B(Cl(G))\cong C_{4}$, where $C_4$ is the {\em cycle} of length $4$. Therefore, it is interesting to see, if one minded so, whether there exists a group $G$ with $B(Cl(G))$ isomorphic to a cycle. Taeri in~\cite{T} has answered this question and  has proved that $B(Cl(G))$ is a cycle if and only if it is the cycle $C_6$ of length six. Moreover, Taeri has classified the groups $G$ with $B(Cl(G))\cong C_6$; indeed, $G\cong A\times \mathrm{SL}_2(q)$ where $A$ is an abelian group and $q\in \{4,8\}$. Since $C_4$ is also the {\em complete bipartite} graph $K_{2,2}$ and since there is no finite group $G$ with $B(Cl(G))\cong K_{2,2}$, Taeri~\cite[Question~1]{T} has asked whether $B(Cl(G))$ can be isomorphic to some complete bipartite graph. In~\cite{HSpiga}, we answered this question and we constructed infinitely many groups $G$ with  $B(Cl(G))\cong K_{2,5}$. However, as far as we are aware, it is not known for which positive integers $n$ and $m$ there exists a finite group $G$ with $B(Cl(G))\cong K_{n,m}$ (let alone a meaningful classification of the groups $G$ with $B(Cl(G))\cong K_{n,m}$).

We conclude this brief discussion on $B(Cl(G))$ recalling that the first author and Iranmanesh~\cite[Theorem~$4.1$]{HI} have classified the groups $G$ where $B(Cl(G))$ is isomorphic to  a {\em path}. This classification was obtained by investigating the combinatorial properties of the bipartite divisor graphs constructed from the product of subsets of positive integers~\cite{HI}.

\smallskip

In this paper we are concerned with the {\em bipartite divisor graph for the set of irreducible complex character degrees}. Given a finite group $G$,  we let $\Irr(G)$ be the set of the irreducible complex characters of $G$, we let $\cd(G):=\{\chi(1)\mid \chi\in \Irr(G)\}$ and we let $\cd(G)^*:=\cd(G)\setminus\{1\}$. Finally, we let $B(G)$ denote the bipartite divisor graph for the set of integers $\cd(G)^*$. We recall that the vertex set is the disjoint union of the set of prime numbers dividing some element of $\cd(G)^*$ and $\cd(G)^*$ itself, where we declare the prime $p$ to be adjacent to the character degree $m$ if and only if $p$ divides $m$.

The scope of this paper is to outline some main results on $B(G)$. We feel that future research on $B(G)$ might benefit from this because these results are scattered over a number of papers (cf.~\cite{H,Hregular, mus, moo4, moo5}). During this process, we are able to improve some of these results. Moreover, along the way, we leave some problems and questions.

In Section~\ref{sec:special graphical shapes}, we investigate the groups $G$ where $B(G)$ is in a certain class of graphs (paths, union of paths, cycles and complete bipartite). In Section~\ref{sec:reg}, we study the groups $G$ where $B(G)$ is a regular graph, that is, all vertices of $B(G)$ have the same valency. Finally, in Section~\ref{sec:bounded}, we study the groups $G$ with $B(G)$ having at most $6$ vertices. We proceed by discussing the main results that have been proved already and (in several occasions) by improving some of this work.

\subsection{Notation}
All groups and graphs in our paper are finite.
We denote by $\g\cd(m,n)$ the {\em greatest common divisor} of the integers $m$ and $n$. Given a prime number $p$, we let
$n_p$ be the {\em $p$-part} of $n$, that is, the largest power of $p$ dividing the integer $n$. Similarly, we denote by $n_{p'}$ the {\em $p'$-part} of $n$, that is, $n_{p'}:=n/n_p$. We let $\pi(n)$ denote the set of all prime divisors of the natural number $n$.

Given a graph $\mathcal{G}$, we let $V(\mathcal{G})$ denote the {\em vertex set}, we let $E(\mathcal{G})$ denote the {\em edge set}, we let $n(\mathcal{G})$ denote the number of {\em connected components} and we let $o(\mathcal{G})$ denote the cardinality of $V(\mathcal{G})$. The diameter of $\mathcal{G}$, denoted  by $\diam(\mathcal{G})$, is the maximum of the diameters of the connected components of $\mathcal{G}$.  If $\mathcal{G}$ is disconnected and $\mathcal{G}_{1},\ldots,\mathcal{G}_{n}$ are  the connected components of $\mathcal{G}$, then we write $\mathcal{G}:=\mathcal{G}_{1}+\cdots +\mathcal{G}_{n}$. By {\em length of a path} or {\em a cycle}, we mean the number of edges in the path or in the cycle. Also, by $P_{n}$ and $C_{n}$, we mean a path of length $n$ and a cycle of length $n$, respectively. A complete graph on $n$ vertices and a complete bipartite graph on $(m,n)$ vertices are denoted by $K_{n}$ and $K_{m,n}$, respectively.

Given a finite group $G$, we let $\pi(G)$ be the set of all {\em prime divisors of the order} of $G$. As usual, we write $dl(G)$ and $h(G)$ to denote the {\em derived length} and the {\em Fitting height} of $G$, respectively. We denote the {\em first and second Fitting subgroups} of $G$ by $\F G$ and ${\bf F}_2(G)$, respectively.  Other notations throughout the paper are standard and should cause no confusion.

We let $\rho(G)$ be the set of all {\em prime numbers dividing some element of} $\cd(G)^*$. The graphs that we use in this paper are:
\begin{description}
\item[Prime graph $\Delta(G)$]
\begin{align*}
V(\Delta(G))&:=\rho(G),\\
E(\Delta(G))&:=\{\{p,q\}\mid p,q\in \rho(G), p\ne q, pq \textrm{ divides some element of }\cd(G)\};
\end{align*}
\item[Common divisor graph $\Gamma(G)$]
\begin{align*}
V(\Gamma(G))&:=\cd(G)^*,\\
E(\Gamma(G))&:=\{\{m,k\}\mid m,k\in \cd(G)^{*}, m\ne k, \g\cd(m,k)\neq 1\};
\end{align*}
\item[Bipartite divisor graph $B(G)$]
\begin{align*}
V(B(G))&:=\rho(G)\amalg \cd(G)^{*}\, (\textrm{disjoint union}),\\
E(B(G))&:=\{\{p,m\}\mid p\in\rho(G),m\in \cd(G)^{*}, p\textrm{ divides }m\}.
\end{align*}
\end{description}

\subsection{Notation in the figures of this paper}
We have consistently drawn all figures of this paper so that the vertices in the lower part of the picture are light blue and are the elements of $\rho(G)$ and the vertices in the upper part of the picture are blue and are the elements of $\cd(G)^{*}$.

\section{Groups whose bipartite divisor graphs have special shapes}
\label{sec:special graphical shapes}

One of the main questions that naturally arises in this area of research is classifying those groups whose bipartite divisor graphs have special shapes. In ~\cite{H}, the first author of this paper discussed the cases where the bipartite divisor graph for the set of irreducible character degrees is
\begin{itemize}
\item a path (see Theorem~\ref{thm:990}),
\item a union of paths for non-solvable groups (see Theorem~\ref{thm:99}), or
\item a cycle (see Theorems~\ref{thm:55}).
\end{itemize}
In this section, one the one hand, we review and we improve the results in~\cite{H}, on the other hand, we discuss the algebraic structure of a solvable group whose bipartite divisor graph is a union of paths.

In our analysis we use the classification into {\em six types} of Mark Lewis~\cite{ML2} of the solvable groups whose degree graph is disconnected. Lewis has named these classes Type~1--6 and for each of these types he has given a detailed description in~\cite[Lemmas~$3.1$--$3.6$]{ML2}. Except for the proof of Theorem~\ref{thm:P4}, we do not need this full classification here. We assume that the reader is broadly familiar with these types, however we highlight below the following properties tailored to our needs.
 \begin{remark}~\label{rem:20}{\rm
  Let $X$ be a solvable group with $\Delta(X)$ disconnected. Then $\Delta(X)$  has two connected components. Moreover, the following hold.
 \begin{itemize}
 \item If $X$ is of type $1$, $2$, $3$, or $5$, then at least one of the connected components of $\Delta(X)$ has cardinality $1$. Thus, if each connected component of $\Delta(X)$ has at least two vertices, then $X$ is a group of type $4$ or $6$. The converse is not true: there are some groups of type $4$ having prime graph consisting of two isolated vertices. For example, the group
\[\texttt{SmallGroup(168,43)}\cong \mathrm{A}\Gamma\mathrm{L}(1,8)\]
 in the ``SmallGroup'' library in GAP~\cite{GAP4} is of type $4$ and its prime graph has vertex set $\{3,7\}$. Moreover, if $X$ is a group of type $6$, then $X$ has a normal Sylow $p$-subgroup and $\Delta(X)$ has a connected component consisting of $\pi([{\bf F}_{2}(X):\F X])\cup\{p\}$ and this set has cardinality greater than $1$.
\item If $X$ is a group of type $1$, then $h(X)=2$, while for all the other types $h(X)\ge 3$, see~\cite[Lemma~4.1]{ML2}.
\item The group $X$ has a normal non-abelian Sylow subgroup if and only if $X$ is of type $1$ or $6$.
\item If $X$ is a group of type $5$, then $\{1,2,2^{a}+1\}\subseteq \cd(X)$, for some positive integer $a$.
\item If $X$ is a group of type $2$, then $\cd(X)=\{1,2,3,8\}$ and if $X$ is of type $3$, then $\cd(X)=\{1,2,3,4,8,16\}$.
 \end{itemize}}
\end{remark}

\subsection{The case where the bipartite divisor graph is a path}
Let $G$ be a finite group. In~\cite{H} it is proved that $B(G)$ has diameter at most seven and this upper bound is the best possible. In the special case that $B(G)$ is a path of length $n$, ~\cite[Proposition 2]{H} improves this bound by showing  that $n\leq 6$; moreover, $G$ is solvable and $dl(G)\leq 5$. The following theorem gives a more detailed description of $G$.
\begin{theorem}[{{See,~\cite{H}}}]~\label{thm:990}
Let $G$ be a finite group with $B(G)$ a path of length $n$. Then,  one of the following occurs:
\begin{itemize}
\item[(i)] $G$ has an abelian normal subgroup $N$ such that $\cd(G)=\{1,[G : N]\}$ and $G/N$ is abelian. Furthermore, $n\in\{1,2\}$.
\item[(ii)] There exist normal subgroups $N$ and $K$ of $G$ and a prime number $p$ with the following properties:
  \begin{itemize}
    \item[(iia)] $G/N$ is abelian;
    \item[(iib)] $\pi(G/K)\subseteq\rho(G)$;
    \item[(iic)] either $p$ divides all the non-trivial irreducible character degrees of $N$  (this implies that $N$ has a normal $p$-complement), or $\cd(N)=\{1, l, k,h/m\}$, where $\cd(G)=\{1, m, h, l, k\}$.
  \end{itemize}
  Furthermore, $n\in\{4,5,6\}$.
\item[(iii)] $\cd(G)=\{1,p^{\alpha},q^{\beta},p^{\alpha}q^{\beta}\}$, where $p$ and $q$ are distinct primes and $\alpha,\beta$ are positive integers. Thus $n=4$.
\item[(iv)] There exists a prime $s$ such that $G$ has a normal $s$-complement $H$. Either $H$ is abelian and $n\in\{1,2\}$ or  $H$ is non-abelian and
    \begin{itemize}
    \item[(iva)] $\cd(G)=\{1,h,hl\}$, for some positive integers $h$ and $l$, and $n=3$;
    \item[(ivb)] $n=4$ and $G/H$ is abelian. Either $\cd(H)=\{[H:\F H]\}\cup \cd(\F H)$ or $\cd(H)=\{1, [{\bf F}_{2}(H):\F H], [H:\F H]\}$. Also $[G:\F G]\in \cd(G)$ and $\cd(\F G)=\{1,h_{s^{'}}\}$, where $[G:\F G]\neq h\in \cd(G)$;
     \item[(ivc)] $n=3$, $G/H$ is abelian, $h:=[G:\F G]\in \cd(G)$, $\F G=P\times A$, where $P$ is a $p$-group for some prime number  $p$, $A\leq \Z G$, $\cd(G)=\cd(G/A)$, and $\cd(P)=\{1, m_{s^{'}}\}$ for $h\neq m\in \cd(G)$.
    \end{itemize}
\end{itemize}
\end{theorem}
In Theorem~\ref{thm:990}, the description of the groups $G$ with $B(G)\cong P_n$ and $n\le 3$ is rather good. (For instance, when $n=1$, or $n=2$ and $|\rho(G)|=2$, the group $G$ has a unique non-trivial character degree. These groups are classified by the work in~\cite[Chapter~12]{IS} and~\cite{BH}. Similarly, when $|\cd(G)^*|=2$, a great deal of information is in~\cite{N}.) However, when $n\ge 4$, the information on $G$ is not yet very satisfactory.

For the time being, we focus on the case $n=4$, that is, $B(G)\cong P_{4}$. This situation may arise from two different cases: either $|\rho(G)|=2$ or $|\rho(G)|=3$. Both cases are possible as it is shown in Table~\ref{tab:p4}. Examples of the first case are rather elementary, it suffices to take the direct product $G:=A\times B$, where $A$ and $B$ are both groups having a unique non-trivial character degree $p_a^\alpha$ and $p_b^\beta$, respectively, where $p_a$ and $p_b$ are distinct prime numbers. It is rather more intriguing to construct examples of the second kind, as witnessed by the fact that the smallest finite group $G$ with $B(G)\cong P_4$ and $\rho(G)=3$ has cardinality $960$. Therefore, we now look more closely to this case.

\begin{table}[ht]
\caption{Examples of $B(G)=P_{4}$}
\centering
\begin{tabular}{|c|c|c|}
\hline $G$ & $|\rho(G)|$ & $\cd(G)$ \\
\hline $\mathrm{Sym}(3)\times \mathrm{Alt}(4)$ & $2$ & $\{1,2,3,6\}$ \\ \hline $\texttt{SmallGroup(960,5748)}$ & $3$ & $\{1,12,15\}$ \\
\hline
\end{tabular}
 \label{tab:p4}
\end{table}

Assume that $G$ is a finite group with $B(G)\cong P_{4}$ and $|\rho(G)|=3$. Let $\rho(G):=\{p,q,r\}$ and let $\alpha,\beta,\gamma,\delta$ be positive integers such that
$\cd(G):=\{1,p^\alpha q^\beta,r^\gamma q^\delta\}$. Since every non-linear character degree of $G$ is divisible by the prime $q$, we deduce (from a celebrated theorem of Thompson~\cite[(12.2)]{IS}) that $G$ has  a normal $q$-complement $L$. Let $Q$ be a Sylow $q$-subgroup of $G$. Thus $G$ equals the semidirect product $L\rtimes Q$. 
Let $\theta\in \Irr(L)$ and let $\chi\in\Irr(G)$ with $\langle \chi_L,\theta\rangle\ne 0$. Then, from Clifford theory, $\chi_L=e(\theta_1+\cdots+\theta_t)$, where $\theta_1,\ldots,\theta_t$ are the conjugates of $\theta$ under $G$. As $\chi(1)\in \{1,p^\alpha q^\beta,r^\gamma q^\delta\}$ and $e,t$ are divisors of $|G:L|=|Q|$ by~\cite[(11.29)]{IS}, we deduce
\begin{equation}\label{eq:2}
\cd(L)=\{1,p^\alpha,r^\gamma\}.
\end{equation}
Therefore $\Delta(L)$ is a disconnected graph with two isolated vertices. As $L$ is solvable, Remark~\ref{rem:20} implies that $L$ is a group of type $1$, $4$ or $5$ in the sense of Lewis. In particular, if $L$ has a non-abelian normal Sylow subgroup, then $L$ is of type $1$. An example of this case is in Table~\ref{tab:p4}, which we now discuss.
\begin{example}~\label{exam: P4}{\em
Let  $G=\texttt{SmallGroup}(960,5748)$. Then $\cd(G)=\{1,12,15\}$ and
 $$G\cong  (\mathbb{Z}_{2}\times\mathbb{Z}_{2}).(\mathbb{Z}_{2}\times\mathbb{Z}_{2}\times\mathbb{Z}_{2}\times\mathbb{Z}_{2})\rtimes\mathbb{Z}_{15}$$ Indeed, $G$ has a normal Sylow $2$-subgroup $P$ with $|P|=2^6=64$, $\Z P=P'=\Phi(P)$ and $|\Z P|=4$.

Using the notation that we have established above, $q=3$, $L$ is the normal $3$-complement of $G$ and $\cd(L)=\{1,4,5\}$. As $\Delta(L)\cong K_{1}+K_{1}$ and $L$ has a normal non-abelian Sylow subgroup, Remark~\ref{rem:20} implies that $L$ is of type $1$.}
 \end{example}

Motivated by Example~\ref{exam: P4}, the following theorem verifies that  $L$ is always a group of type $1$ and  explains in part the elusiveness of the groups $G$ with $B(G)\cong P_4$ and $|\rho(G)|=2$.
 \begin{theorem}~\label{thm:P4}
Suppose that $\cd(G_0)=\{1,p^\alpha q^\beta,r^\gamma q^\delta\}$. Then there exists $A\le \Z {G_0}$ such that for the factor group $G:=G_0/A$ the following holds (replacing $r$ with $p$ if necessary):
\begin{itemize}
\item[(i)] $G$ contains a normal Sylow $p$-subgroup $P$;
\item[(ii)]$P=G'=\F G$;
\item[(iii)]$P$ is semiextraspecial and $G/P$ is cyclic ($P$ is called semiextraspecial if, for all maximal subgroups $N$ of $\Z P$, the factor group $P/N$ is extraspecial);
\item[(iv)]$G/P'$ is a Frobenius group with Frobenius kernel $P/P'$;
\item[(v)]$P<\cent G{P'}<G$;
\item[(vi)]$G/\cent G{P'}$ acts as a Frobenius group on $P'$;
\item[(vii)]$\cd(G)=\cd(G_0)$, where $p^{2\alpha}=|P:P'|$ and $q^{\beta}=|G:\cent G {P'}|$;
\item[(viii)]if $L$ is the $q$-complement of $G$, then $L$ is of type $1$ in Lewis sense.
\end{itemize}
\end{theorem}
\begin{proof}
The proof follows applying Theorem~$5.6$ in~\cite{N} in our context.
\end{proof}

Given a finite group $G$, it is easy verify that, if $B(G)$ is a path, then both $\Delta(G)$ and $\Gamma(G)$ are paths. In particular, if $B(G)\cong  P_{6}$, then $\Gamma(G)\cong  P_{3}$ and hence $\Gamma(G)$ has diameter three. The converse is not always true, as we show in the following example.
\begin{example}{\em
We recall a construction from~\cite{shafiei}.
 Let $p$, $q$, and $r$ be three, not necessarily distinct, primes with $q\neq r$ such that $q$ and $r$ do not divide $p^{qr}-1$, let $V$ be the additive group of the field $\mathbb{F}_{p^{qr}}$ of order $p^{qr}$, let $S$ be the Galois group of the field extension $\mathbb{F}_{p^{qr}}/\mathbb{F}_p$, let $C$ be the cyclic subgroup of  order $\frac{p^{qr}-1}{p^{r}-1}$ of the multiplicative group of $\mathbb{F}_{p^{qr}}^*$ and  let $G:=(V\rtimes C)\rtimes S$. Then $$\cd(G)=\left\{1,q,qr,\frac{p^{qr}-1}{p^{r}-1},r\frac{p^{qr}-1}{p^{r}-1}\right\}.$$ Thus $B(G)$ is a path if and only if $\frac{p^{qr}-1}{p^{r}-1}$ is a prime power. Lemma~$3.1$ in~\cite{shafiei} yields that, if $\frac{p^{qr}-1}{p^{r}-1}$ is a prime power, then either $r$ is a power of $q$, or $p=q=2$ and $r=3$.

The first case is not possible because by hypothesis $r$ and $q$ are distinct primes. In the second case, $r=3$ is a divisor of $p^{qr}-1=63$, which contradicts our hypothesis. Thus $\frac{p^{qr}-1}{p^{r}-1}$ is not a prime power and $B(G)$ is not a path.}
\end{example}

Next, we prove the existence of  groups $G$ with $B(G)\cong  P_{n}$, for $n\in\{5,6\}$. This  answers  Question~1 in~\cite{H}.
\begin{example}\label{ex:5}{\rm
When $n=5$, we give infinitely many groups $G$ with $B(G)\cong P_5$. Our example is due to P\'{e}ter P\'{a}l P\'{a}lfy, and we gratefully acknowledge his contribution.

Let $r$ be an odd prime, let $p$ be a prime with $p\equiv 1\pmod {2r}$ and let $P$ be an extraspecial group of order $p^{3}$ with exponent $p$. (Observe that the existence of infinitely many primes $p$ follows, for instance, from Dirichlet theorem on primes in arithmetic progression.) The group $P$ has the following presentation:
$$P=\langle x_{1},x_{2}\mid x_{1}^{p}=x_{2}^{p}=[x_1,x_2]^p=[x_1,[x_1,x_2]]=[x_2,[x_1,x_2]]=1\rangle.$$
Since $2r$ divides $p-1$, there exists $\alpha\in \mathbb{Z}$ such that $\alpha $ has order $2r$ in  the multiplicative group of $(\mathbb{Z}/p\mathbb{Z})^*$. Set $\beta:=\alpha^{r-1}$ and observe that $\beta$ has order $r$ in the multiplicative group $(\mathbb{Z}/p\mathbb{Z})^*$ because $\gcd(2r,r-1)=2$.  From the presentation of $P$, we see that the mapping
$$x_{1}\mapsto x_{1}^{\alpha},\quad x_{2}\mapsto x_{2}^{\beta}$$
defined on the generators of $P$ extends to an automorphism of $P$ of order $2r$, which we denote by $y$.
Set $G:=P\rtimes \langle y\rangle$ with respect to the above action. We have
$$[x_1,x_2]^y=[x_1^y,x_2^y]=[x_1^\alpha,x_2^\beta]=[x_1,x_2]^{\alpha\beta}=[x_1,x_2]^{\alpha\alpha^{r-1}}=[x_1,x_2]^{\alpha^r}.$$ Since $\alpha^r$ has order $2$ modulo $p$, we obtain $[x_1,x_2]^y=[x_1,x_2]^{-1}$ and hence $y$ acts on $\Z P$ by inverting its elements. From this, it is easy to see that $\cd(G)=\{1,r,2r,2p\}$, so $B(G)$ is a path of length five.
}
\end{example}

\begin{example}\label{ex:6}{\rm
When $n=6$, our construction is based on some preliminary theoretical work and then its implementation in a computer. Here, we report only the outcome of our computations because we are not able to give a general construction.

Let $P$ be the group $\texttt{SmallGroup}(256,3679)$. 
One can check that the group $P$ has nilpotency class $3$, $|P:\gamma_2(P)|=|\gamma_2(P):\gamma_3(P)|=2^3=8$ and $|\gamma_3(P)|=4$. Moreover, each section of the lower central series has exponent $2$. Let $T$ be a  Hall $2'$-subgroup in the automorphism group of $P$. A simple computation yields that $T$ is non-abelian and has cardinality $21$. Let $G$ be the semidirect product $P\rtimes T$. Then $G$ is a solvable group  having cardinality $5\,376=2^8\cdot 3\cdot 7$ and a computation yields $\cd(G)=\{1,3,7,14,24\}.$ Therefore $B(G)$ is a path of length $6$. Furthermore, considering $G$ as a permutation group of degree $32$, it is generated by $\alpha_{1}$, $\alpha_{2}$, and $\alpha_{3}$, where
\begin{align*}
&\text{\small$\alpha_{1}:=(1,2)(3,14,9,20)(4,15,19,27)(5,7)(6,8)(10,21,18,26)(11,22,12,23)(13,16)(17,31,25,29)(24,32,28,30)$},\\
&\text{\small$\alpha_{2}:=(2,27,25,19,21,10,31)(3,29,8,22,17,11,14)(4,20,9,30,16,15,24)(7,23,28,12,26,18,32)$},\\
&\text{\small$\alpha_{3}:=(3,12,30)(4,29,18)(5,13,6)(7,16,8)(9,11,32)(10,19,31)(14,23,24)(15,17,26)(20,22,28)(21,27,25)$}.
\end{align*}
This is the smallest group we managed to construct having $B(G)$ a path of length $6$.}
\end{example}

Now that we have established the existence of a group $G$ with $B(G)\cong P_n$ for $n\in\{5,6\}$,  it would be interesting  to give a classification of this family of groups.

\begin{problem}Give structural information on the finite groups $G$ with $B(G)\cong P_{n}$ for $n\in\{5,6\}$.
\end{problem}

\vskip 0.4 true cm

\subsection{Union of paths}\label{unionpaths}
As it is explained in~\cite[Example 3.4]{N}, given a prime $p$ and two positive integers $k$ and $m$ with $k$ dividing $p^m\pm 1$, there exists a solvable group $PH$ such that $\cd(PH)=\{1,k,p^{m}\}$. (More information on the structure of $G$ is in~\cite{N}, but this is of no concern here.) In particular, when $|\pi(k)|=2$, $B(PH)$ is the first graph in Figure~$\ref{fig: 5}$, and hence $B(PH)$ is a union of paths.
On the other hand, $\cd(M_{10})=\{1,9,10,16\}$ and $\cd(\mathrm{PSL}_2(25))=\{1,13,24,25,26\}$ and hence $B(M_{10})$ and $B(\mathrm{PSL}_2(25))$ have two connected components which are both paths: $B(M(10))=P_2+P_3$ and $B(\mathrm{PSL}_2(25))=P_2+P_5$.

These examples stimulate the curiosity of investigating those finite groups $G$ such that  each connected component of $B(G)$ is a path. For this aim, first we let $G$ be a finite non-solvable group. In the following theorem, we refine the main result of~\cite{H}, cf.~\cite[Theorem~6]{H}.

\begin{theorem}~\label{thm:99}
Let $G$ be a finite non-solvable group with $B(G)$ a union of paths. Then $B(G)$ is disconnected, $B(G)$ has $2$ or $3$ connected components and $|\cd(G)|\in \{4,5\}$.

Moreover, we have one of the following cases:
\begin{itemize}
\item[(i)]$n(B(G))=2$, $|\cd(G)|=4$, $G$ has a normal subgroup $U$ such that $U\cong \mathrm{PSL}_2(q)$ or $U\cong \mathrm{SL}_2(q)$ for some odd $q\ge 5$ and, if $C:=\cent G U$, then $C\le \Z G$ and $G/C\cong \mathrm{PGL}_2(q)$. Thus $\cd(G)=\{1,q,q-1,q+1\}$.
\item[(ii)]$n(B(G))=2$, $|\cd(G)|=4$, $G$ has a normal subgroup of index $2$ that is a direct product of $\mathrm{PSL}_2(9)$ and a central subgroup $C$. Furthermore, $G/C\cong M_{10}$ and $\cd(G)=\{1,9,10,16\}$.
\item[(iii)]$n(B(G))=2$, $|\cd(G)|=5$, $G$ has a solvable normal subgroup $V$ such that $G/V$ is almost simple and isomorphic to either $\mathrm{PSL}_2(q)$, or $\mathrm{PGL}_2(q)$, or $\mathrm{P}\Gamma\mathrm{L}_2(2^r)$ (for some prime number $r$), $M_{10}$,
$\mathrm{P}\Sigma\mathrm{L}_2(9)$, or $\mathrm{PGL}_2(3^{2f'}).2$ (for some $f'\ge 1$).
\item[(iv)]$n(B(G))=3$, $\cd(G)=\{1,2^n,2^n-1,2^n+1\}$ and $G\cong  \mathrm{PSL}_2(2^{n})\times A$, where $A$ is an abelian group and $n\geq 2$.
\end{itemize}
\end{theorem}
\begin{proof}
From Theorem~\ref{thm:990}, the non-solvability of $G$ implies that $B(G)$ is disconnected. By~\cite[Theorem 6.4]{L}, we have $n(\Delta(G))\le 3$ and,  by~\cite{IP}, we have $n(B(G))=n(\Delta(G))$. Thus $B(G)$ is disconnected with at most three connected components.

When $n(B(G))=3$, the result follows from~\cite{H} and we obtain part~(iv). Suppose then $n(B(G))=2$. From~\cite{H}, we deduce $\cd(G)\in \{4,5\}$. When $\cd(G)=4$, the result follows from~\cite[Theorem~A]{GA} and we obtain parts~(i) and~(ii).  Finally, suppose that $\cd(G)=5$. The finite groups having $5$ (or $6$) character degrees are classified in~\cite[Corollary~C]{LZ}. From this classification, we see that $G$ contains a normal solvable subgroup $V$ with $G/V$ almost simple with socle $\mathrm{PSL}_2(p^f)$ (with $p^f\ge 4$), or $\mathrm{PSL}_3(4)$, or $^{2}B_2(2^{2m+1})$ (with $m\ge 1$). The cases $\mathrm{PSL}_3(4)$ and $^2{B}_2(2^f)$ do not arise here because any almost simple group having socle one of these two groups has $6$ irreducible complex character degrees. In particular, $G/V$ is almost simple with socle $\mathrm{PSL}_2(p^f)$. Now, the actual structure of $G/V$ can be inferred from the work of White~\cite[Theorem~A]{white}. Indeed, White computes explicitly the character degrees of each almost simple group $X$ having socle $\mathrm{PSL}_2(p^f)$. In part~(iii), we have selected the groups $X$ with $B(X)$ a union of paths.
\end{proof}

Observe that the converse of Theorem~\ref{thm:99} does not hold; for instance, $B(\textrm{PSL}_2(2^n))$ consists of three connected components for each value of $n$, however these connected components are not necessarily paths (this depends on number-theoretic questions concerning the factorization of $2^n+1$ and $2^n-1$).

\vskip 0.3 true cm

Now, we let $G$ be a {\em solvable} group with $B(G)$ disconnected and union of paths (the connected case was discussed in the previous section). As $G$ is solvable, we have $n(B(G))=2$ and $|\rho(G)|\geq 2$. Here we describe the structure of $G$ using  the  types introduced by Lewis. By Remark~\ref{rem:20}, type $3$ does not occur. As $\Delta(G)$ is triangle-free, the main result in~\cite[Lemma~$2.2$]{Tong-viet} yields $|\rho(G)|\leq 4$. Using $n(B(G))=2$ and $2\le |\rho(G)|\le 4$, a simple case-by-case analysis gives that $B(G)$ is one of the graphs drawn in Figures~\ref{fig: 4},~\ref{fig: 5} or~\ref{fig: 6}. We have tabulated some information on these graphs and the groups yielding these graphs inTable~\ref{tab:TTCR}. This table consists of three columns: the first column is one of the graphs $\Gamma$ in Figures~\ref{fig: 4},~\ref{fig: 5} or~\ref{fig: 6}, the second column are the Lewis types $X\subseteq\{1,2,3,4,5,6\}$ of the groups $G$ with $\Gamma\cong B(G)$ and in the third column  we exhibit (for each $x\in X$) a group of Lewis type $x$ and with $\Gamma\cong B(G)$. The task in the rest of this section is proving the correctedness of Table \ref{tab:TTCR}.

\begin{lemma}\label{lemma:new}
Let $G$ be a group with $B(G)$ isomorphic to the second graph in Figure~$\ref{fig: 4}$. Then $G$ is not of Lewis type $4$.
\end{lemma}
\begin{proof}
We argue by contradiction and we suppose that $G$ is of type $4$. We use the notation established in~\cite{ML2} for the groups of type $4$ and we suppose that the reader is familiar with basic properties about groups in this class (in particular with Example~$2.4$ and Lemma~$3.4$ in~\cite{ML2}). Since $\cd(G)=\cd(G/\Z G)$ by~\cite[Lemma~$3.4$]{ML2}, we may suppose $\Z G=1$. In particular, $G=V\rtimes H$ and $H$ acts irreducibly as a linear group on the elementary abelian $p$-group $V$.

Recall, from~\cite{ML2}, that $K:=\F H$, $m:=|E:K|>1$, $K$ is cyclic, $|V|=q^m$ where $q$ is a power of the prime $p$, and $(q^m-1)/(q-1)$ divides $|K|$. Now, $\cd(G|V)=\{|K|\}$ and $\cd(G/V)$ consists of $1,m$ and eventually some other divisors of $m$, see~\cite[Lemma~$3.4$]{ML2}.

As $B(G)$ is isomorphic to the second graph in Figure~\ref{fig: 4}, we deduce that $|K|$ is a prime power (say $|K|=r^\ell$ for some prime number $r$ and some positive integer $\ell$), $m$ is a prime power (say $m=s^t$ for some prime number $s$ and some positive integer $t$) and $\cd(G/V)=\{1,s^t,s^{t'}\}$ for some $0<t'<t$. In particular, $t\ge 2$.

Recall that, given two positive integers $x$ and $y$, the prime number $z$ is said to be  a primitive prime divisor of $x^y-1$ if $z$ divides $x^y-1$, but (for every $i\in \{1,\ldots,y-1\}$) $z$ does not divide $x^i-1$.

Let $z$ be a primitive prime divisor of $q^{m}-1=q^{s^t}-1$. As $(q^m-1)/(q-1)$ divides $|K|=r^\ell$, we deduce $z=r$. As $(q^{s^{t-1}}-1)/(q-1)$ divides $(q^{s^t}-1)/(q-1)$ and $t\ge 2$, we deduce that $r$ divides $q^{s^{t-1}}-1$, contradicting the fact that $r$ is  a primitive prime divisor of $q^m-1=q^{s^t}-1$. Therefore $q^{s^t}-1$ has no primitive prime divisors. From a celebrated theorem of Zsigmondy~\cite{zsigmondy}, we deduce that either $q=2$ and $s^t=6$, or $s^t=2$ and $q$ is a Mersenne prime, however in both cases we obtain a contradiction.
\end{proof}

\begin{lemma}\label{lemma:new1}
Let $G$ be a group with $B(G)$ isomorphic to the second graph in Figure~$\ref{fig: 5}$, to the first or to the third graph in Figure~$\ref{fig: 6}$, or to the third graph in Figure~\ref{fig: 2}. Then $G$ is not of Lewis type $6$.
\end{lemma}
\begin{proof}
We argue by contradiction and we suppose that $G$ is of type $6$. We use the notation established in~\cite{ML2} for the groups of type $6$ and we suppose that the reader is familiar with basic properties about groups in this class (in particular with Example~$2.6$ and Lemma~$3.6$ in~\cite{ML2}). From~\cite[Lemma~$3.6$~(v),(vi)]{ML2}, $\cd(G)=\cd(G/A')\cup\cd(G|A')$, where $G/A'$ is a group of Lewis type $4$ and $\cd(G|A')$ consists of degrees that divide $|P||E:F|$ and are divisible by $p|B|$, (observe that $p$ and $|B|$ are relatively prime). These facts together imply that $B(G/A')$ is a union of paths and is obtained by deleting some blue vertices of $B(G)$ having at least two light blue neighbors. If $B(G)$ is the third graph in Figure~\ref{fig: 6}, then we can delete only one vertex from  $B(G)$ in order to obtain $B(G/A')$. However, the resulting graph has three connected components, contradicting the fact that $G/A'$ is solvable. If $B(G)$ is the first graph in Figure~\ref{fig: 6}, then (again) we can delete only one vertex from  $B(G)$ in order to obtain $B(G/A')$. Therefore, the resulting graph is the second graph in Figure~\ref{fig: 4}. However, Lemma~\ref{lemma:new} excludes this possibility. If $B(G)$ is the second graph in Figure~\ref{fig: 5}, then  we can delete only one vertex from  $B(G)$ in order to obtain $B(G/A')$. Now, the resulting graph is connected, contradicting the fact that it is a graph of Lewis type $4$ and hence disconnected. Finally, if $B(G)$ is the third graph in Figure~\ref{fig: 2}, then we can delete only one vertex from $B(G)$ to obtain $B(G/A')$. The result is a connected graph which is impossible.
\end{proof}

\begin{example}\label{exa:new}
{\rm Let $P$ be the group $\mathtt{SmallGroup}(2\,187,9\,308)$. A computation shows that $\mathrm{Aut}(P)$ contains a cyclic subgroup $H$ of order $10$ with the property that $\cd(P\rtimes H)=\{1,2,10,27\}$ and $P\rtimes H$ has Lewis type $1$. In particular, $G:=P\rtimes H$ is a finite soluble group of order $21\,870$ of Lewis type $1$ with $\cd(G)=\{1,2,10,27\}$ and with  $B(G)$  isomorphic to the third graph in Figure~\ref{fig: 5}.

This example was constructed with the help of a computer after deducing some preliminary theoretical properties. Incidentally, this is the smallest example we managed to find, but we were not
 able to prove that it is indeed the example of smallest cardinality with $B(G)$ isomorphic to the third graph in Figure~\ref{fig:  5}.}
\end{example}

\begin{example}\label{exa:newnew}
{\rm Let $G$ be the polycyclic group with presentation $G:=\langle x_1,\ldots,x_{15}\mid R\rangle$, where the set of polycyclic relations $R$ are given by
\begin{align*}
&x_1^2 = x_{15},\,
    x_2^2 = x_{15},\,
    x_3^3 = 1,\,
    x_4^{11} = 1,\,
    x_5^2 = x_{15},\,
    x_6^2 = 1,\,
    x_7^2 = 1,\,
    x_8^2 = 1,\,
    x_9^2 = 1,\,
x_{10}^2 = x_{15},\,
x_{11}^2 = 1,\\
&    x_{12}^2 = 1,\,
    x_{13}^2 = 1,\,
    x_{14}^2 = 1,\,
    x_{15}^2 = 1,\,
    x_2^{x_1} = x_2 \cdot x_{15},\,
    x_3^{x_2} = x_3^2,\,
    x_4^{x_2} = x_4^{10},\,
    x_5^{x_2} = x_5 \cdot x_{15},\\
&x_5^{x_3} = x_5 \cdot x_7 \cdot x_{10} \cdot x_{11} \cdot x_{12},\,
    x_5^{x_4} = x_{10},\,
    x_6^{x_3} = x_6 \cdot x_8 \cdot x_{11} \cdot x_{12} \cdot x_{13},
    x_6^{x_4} = x_{11},\,
x_7^{x_3} = x_7 \cdot x_9 \cdot x_{12} \cdot x_{13} \cdot x_{14},\\
&    x_7^{x_4} = x_{12},\,
    x_8^{x_3} = x_5 \cdot x_7 \cdot x_8 \cdot x_{10} \cdot x_{12} \cdot x_{13} \cdot x_{14},\,
    x_8^{x_4} = x_{13},\,
    x_9^{x_3} = x_6 \cdot x_8 \cdot x_9 \cdot x_{10} \cdot x_{11} \cdot x_{12} \cdot x_{13} \cdot x_{14},\\
&    x_9^{x_4} = x_{14},\,
    x_{10}^{x_2} = x_7 \cdot x_8 \cdot x_{10} \cdot x_{15},\,
    x_{10}^{x_3} = x_5 \cdot x_6 \cdot x_7 \cdot x_{12} \cdot x_{15},\,
    x_{10}^{x_4} = x_5 \cdot x_{12} \cdot x_{13},\,
    x_{10}^{x_5} = x_{10} \cdot x_{15},\,\\
&    x_{10}^{x_6} = x_{10} \cdot x_{15}, \,
    x_{10}^{x_7} = x_{10} \cdot x_{15}, \,
    x_{10}^{x_8} = x_{10} \cdot x_{15}, \,
    x_{11}^{x_2} = x_8 \cdot x_9 \cdot x_{11}, \,
    x_{11}^{x_3} = x_6 \cdot x_7 \cdot x_8 \cdot x_{13} \cdot x_{15},\,\\
&    x_{11}^{x_4} = x_6 \cdot x_{13} \cdot x_{14}, \,
    x_{11}^{x_5} = x_{11} \cdot x_{15}, \,
    x_{11}^{x_6} = x_{11} \cdot x_{15}, \,
    x_{11}^{x_7} = x_{11} \cdot x_{15}, \,
    x_{12}^{x_2} = x_5 \cdot x_7 \cdot x_9 \cdot x_{12}, \,\\
&    x_{12}^{x_3} = x_7 \cdot x_8 \cdot x_9 \cdot x_{14} \cdot x_{15}, \,
    x_{12}^{x_4} = x_7 \cdot x_{10} \cdot x_{12} \cdot x_{14}, \,
    x_{12}^{x_5} = x_{12} \cdot x_{15}, \,
    x_{12}^{x_6} = x_{12} \cdot x_{15}, \,\\
&    x_{13}^{x_2} = x_5 \cdot x_6 \cdot x_7 \cdot x_8 \cdot x_{13} \cdot x_{15}, \,
    x_{13}^{x_3} = x_5 \cdot x_7 \cdot x_8 \cdot x_9 \cdot x_{10} \cdot x_{12}, \,
    x_{13}^{x_4} = x_8 \cdot x_{10} \cdot x_{11} \cdot x_{12} \cdot x_{13} \cdot x_{15}, \,\\
&    x_{13}^{x_5} = x_{13} \cdot x_{15}, \,
    x_{13}^{x_9} = x_{13} \cdot x_{15}, \,
    x_{14}^{x_2} = x_6 \cdot x_7 \cdot x_8 \cdot x_9 \cdot x_{14} \cdot x_{15}, \,
    x_{14}^{x_3} = x_5 \cdot x_6 \cdot x_7 \cdot x_8 \cdot x_9 \cdot x_{11} \cdot x_{13} \cdot x_{15}, \,\\
&    x_{14}^{x_4} = x_9 \cdot x_{11} \cdot x_{12} \cdot x_{13} \cdot x_{14} \cdot x_{15}, \,
    x_{14}^{x_8} = x_{14} \cdot x_{15}, \,
    x_{14}^{x_9} = x_{14} \cdot x_{15}.
\end{align*}
The group $G$ has order $270\, 336=2^{13}\cdot 3\cdot 11$; moreover, $G$ contains a Hall $2'$-subgroup $K$ with $K$ cyclic and $G$ contains a normal $2$-subgroup $Q$ with $|G:QK|=2$. Furthermore, $Q/\gamma_2(Q)$ is elementary abelian of order $2^{10}=1\,024$, $\gamma_2(Q)$ is elementary abelian of order $2^2=4$ and $K$ centralizes $\gamma_2(Q)$. One might check that $G$ has Lewis type $5$. Finally, $\cd(G)=\{1,2,33,64\}$ and hence $B(G)$ is isomorphic to the second graph in Figure~\ref{fig: 5}.

As in Example~\ref{exa:new}, this example was constructed (with some luck) with the help of a computer after deducing some preliminary theoretical properties.}
\end{example}
\begin{example}\label{exa:newnewnew}
{\rm Let $G$ be the polycyclic group with presentation $G:=\langle x_1,\ldots,x_{16}\mid R\rangle$, where the set of polycyclic relations $R$ are given by
\begin{align*}
&x_1^3 = 1,
    x_2^{11} = 1,
    x_3^2 = x_{13},
    x_4^2 = x_{13} \cdot x_{16},
    x_5^2 = x_{13},
    x_6^2 = x_{13},
    x_7^2 = 1,
    x_8^2 = 1,
    x_9^2 = x_{16},
    x_{10}^2 = x_{13},
    x_{11}^2 = x_{13},\\
&    x_{12}^2 = x_{16},
    x_{13}^2 = 1,
    x_{14}^2 = 1,
    x_{15}^2 = 1,
    x_{16}^2 = 1,
    x_3^{x_1} = x_4 \cdot x_6 \cdot x_8,
    x_3^{x_2} = x_4 \cdot x_6 \cdot x_7 \cdot x_8 \cdot x_{11} \cdot x_{12},\\
&    x_4^{x_1} = x_5 \cdot x_7 \cdot x_9,
    x_4^{x_2} = x_3 \cdot x_4 \cdot x_6 \cdot x_7 \cdot x_{12},
    x_5^{x_1} = x_6 \cdot x_8 \cdot x_{10},
    x_5^{x_2} = x_3 \cdot x_6 \cdot x_7 \cdot x_9,
    x_5^{x_3} = x_5 \cdot x_{13},\\
&    x_6^{x_1} = x_7 \cdot x_9 \cdot x_{11},
    x_6^{x_2} = x_4 \cdot x_7 \cdot x_8 \cdot x_{10},
    x_6^{x_3} = x_6 \cdot x_{13} \cdot x_{16},
    x_6^{x_4} = x_6 \cdot x_{16},
    x_6^{x_5} = x_6 \cdot x_{16},
    x_7^{x_1} = x_8 \cdot x_{10} \cdot x_{12},\\
&    x_7^{x_2} = x_5 \cdot x_8 \cdot x_9 \cdot x_{11} \cdot x_{13} \cdot x_{16},
    x_7^{x_3} = x_7 \cdot x_{13} \cdot x_{16},
    x_7^{x_4} = x_7 \cdot x_{16},
    x_7^{x_5} = x_7 \cdot x_{13},
    x_7^{x_6} = x_7 \cdot x_{13}, \\
&    x_8^{x_1} = x_3 \cdot x_4 \cdot x_5 \cdot x_6 \cdot x_8 \cdot x_{11},
    x_8^{x_2} = x_6 \cdot x_9 \cdot x_{10} \cdot x_{12} \cdot x_{16},
    x_8^{x_3} = x_8 \cdot x_{16},
    x_8^{x_4} = x_8 \cdot x_{16},
    x_8^{x_6} = x_8 \cdot x_{13} \cdot x_{16},\\
&    x_8^{x_7} = x_8 \cdot x_{13},
    x_9^{x_1} = x_4 \cdot x_5 \cdot x_6 \cdot x_7 \cdot x_9 \cdot x_{12},
    x_9^{x_2} = x_3 \cdot x_4 \cdot x_5 \cdot x_6 \cdot x_7 \cdot x_8 \cdot x_9 \cdot x_{10} \cdot x_{11} \cdot x_{13},
    x_9^{x_3} = x_9 \cdot x_{16}, \\
&    x_9^{x_4} = x_9 \cdot x_{13},
    x_9^{x_5} = x_9 \cdot x_{16},
    x_9^{x_6} = x_9 \cdot x_{13} \cdot x_{16},
    x_9^{x_7} = x_9 \cdot x_{13} \cdot x_{16},
    x_9^{x_8} = x_9 \cdot x_{16},
    x_{10}^{x_1} = x_3 \cdot x_4 \cdot x_7 \cdot x_9 \cdot x_{10},\\
&    x_{10}^{x_2} = x_4 \cdot x_5 \cdot x_6 \cdot x_7 \cdot x_8 \cdot x_9 \cdot x_{10} \cdot x_{11} \cdot x_{12} \cdot x_{16},
    x_{10}^{x_3} = x_{10} \cdot x_{16},
    x_{10}^{x_4} = x_{10} \cdot x_{13} \cdot x_{16},
    x_{10}^{x_5} = x_{10} \cdot x_{13} \cdot x_{16},\\
&    x_{10}^{x_6} = x_{10} \cdot x_{16},
    x_{10}^{x_7} = x_{10} \cdot x_{13},
    x_{10}^{x_9} = x_{10} \cdot x_{13},
    x_{11}^{x_1} = x_4 \cdot x_5 \cdot x_8 \cdot x_{10} \cdot x_{11},
    x_{11}^{x_2} = x_3 \cdot x_4 \cdot x_7 \cdot x_{10} \cdot x_{11} \cdot x_{12} \cdot x_{16}, \\
&    x_{11}^{x_3} = x_{11} \cdot x_{13} \cdot x_{16},
    x_{11}^{x_4} = x_{11} \cdot x_{13},
    x_{11}^{x_6} = x_{11} \cdot x_{13} \cdot x_{16},
    x_{11}^{x_8} = x_{11} \cdot x_{13} \cdot x_{16},
    x_{11}^{x_9} = x_{11} \cdot x_{13},
    x_{11}^{x_{10}} = x_{11} \cdot x_{13} \cdot x_{16},\\
&    x_{12}^{x_1} = x_5 \cdot x_6 \cdot x_9 \cdot x_{11} \cdot x_{12},
    x_{12}^{x_2} = x_3 \cdot x_6 \cdot x_9 \cdot x_{11} \cdot x_{12} \cdot x_{13},
    x_{12}^{x_3} = x_{12} \cdot x_{13},
    x_{12}^{x_4} = x_{12} \cdot x_{13},
    x_{12}^{x_5} = x_{12} \cdot x_{13} \cdot x_{16},\\
&    x_{12}^{x_6} = x_{12} \cdot x_{13} \cdot x_{16},
    x_{12}^{x_8} = x_{12} \cdot x_{13},
    x_{12}^{x_9} = x_{12} \cdot x_{13},
    x_{12}^{x_{10}} = x_{12} \cdot x_{16},
    x_{12}^{x_{11}} = x_{12} \cdot x_{13} \cdot x_{16},
    x_{15}^{x_{14}} = x_{15} \cdot x_{16}.
\end{align*}
The group $G$ has order $540\, 672=2^{14}\cdot 3\cdot 11$; moreover, $G$ contains a Hall $2'$-subgroup $K$ with $K$ cyclic and $G$ contains a normal Sylow $2$-subgroup $P$. Furthermore, $P/\gamma_2(P)$ is elementary abelian of order $2^{12}=2\,048$, $\gamma_2(P)$ is elementary abelian of order $2^2=4$ and $K$ centralizes $\gamma_2(P)$. One might check that $G$ has Lewis type $1$. Finally, $\cd(G)=\{1,32,33,64\}$ and hence $B(G)$ is isomorphic to the second graph in Figure~\ref{fig: 5}.
}
\end{example}

\begin{remark}{\rm
If $|\cd(G)^{*}|=4$, then $B(G)$ is either the first or the second  graph in Figure~\ref{fig: 6} or the third graph in Figure~\ref{fig: 4}. Except for the second graph in Figure~\ref{fig: 6}, as $\Gamma(G)$ has no isolated vertices, by ~\cite[Theorem 5.2]{ML2}, we deduce that $G$ has a normal non-abelian Sylow subgroup. Now Remark~\ref{rem:20} implies that $G$ is a group of type $1$ or $6$ in the sense of Lewis. For the first graph in Figure~\ref{fig: 6} the case of Lewis type $6$ is excluded by Lemma~\ref{lemma:new1}. 
If $B(G)$ is the last graph in Figure~\ref{fig: 4}, then both connected components of $\Delta(G)$ are isolated vertices; so by Remark~\ref{rem:20} and the previous results we conclude that $G$ is a group of type $1$ (see also~\cite[Theorem~$3.1$]{LL}).

If $B(G)$ is the second graph in Figure~\ref{fig: 6}, then $\Gamma(G)\cong  K_{1}+P_{2}$  consists of one isolated vertex and one edge and hence, by ~\cite[Theorem 3.3]{LL}, we deduce that $G$ is either a group of type $1$ or $4$ in the sense of Lewis. If $G$ has no non-abelian normal Sylow subgroup, then ~\cite[Theorem 5.2]{ML2} implies that the prime divisors of the isolated vertex of $\Gamma(G)$ gives the larger component of $\Delta(G)$, which is not the case. Thus $G$ has a non-abelian normal Sylow subgroup. This implies that $G$ is not a group of type $4$ and so it is a group of type $1$.
\vskip 0.3 true cm

Suppose $B(G)$ is one of the first two graphs in Figure~\ref{fig: 4}. As $\Delta(G)$  has two isolated vertices, from Remark~\ref{rem:20}, we conclude that $G$ is neither a group of type $3$ nor of type $6$ in the sense of Lewis. If $B(G)$ is the first graph in Figure~\ref{fig: 4}, then it is a $1$-regular bipartite graph. The structure and the Lewis type of such a group is explicitly explained in Theorem~\ref{thm: reg} below (and we refer the reader to this theorem for a detailed description). Finally, if $B(G)$ is the second graph in Figure~\ref{fig: 4}, then $G$ is a group of type $1$, $2$ or $5$ in the sense of Lewis (type $4$ does not arise because of Lemma~\ref{lemma:new}).
\vskip 0.3 true cm

Suppose $B(G)$ is either the first or the third graph in Figure~\ref{fig: 5}. By Remark~\ref{rem:20}, $G$ is not a group of type $2$ or $3$. If $B(G)$ is the third graph in this figure and $G$ has no non-abelian normal Sylow subgroup, then by ~\cite[Theorem 5.2]{ML2} we conclude that the prime divisors of the isolated vertex of $\Gamma(G)$ lie in a larger component of $\Delta(G)$ which is not the case for this graph. Hence $G$ has a non-abelian normal Sylow subgroup which implies that $G$ is of Lewis type $1$ or $6$.

Suppose $B(G)$ is the first graph in Figure~\ref{fig: 5}. As $|cd(G)^{*}|$ consists of co-prime degrees, with respect to its Fitting height which is either $2$ or $3$, $G$ has one of the 
structures explained in~\cite[Lemma 4.1]{L1998}. By~\cite[Lemma 4.1, Theorem 4.5]{ML2} we have $h(G)=2$ if and only if $G$ is a group of type $1$. While $h(G)=3$, ~\cite[Lemma 4.1(a-iii)]{L1998} implies that $\F G$ is abelian and in particular $G$ has no non-abelian normal Sylow subgroup. Hence by Remark~\ref{rem:20} $G$ is either of Lewis type $4$ or $5$. Suppose $G$ is of Lewis type $5$. Considering the notations in ~\cite[Lemma 3.5]{ML2}, we deduce that $\{1,2,2^{a}+1\}\subseteq\cd(G)$, $\cd(G|Q')\neq\emptyset$ as $Q$ is non-abelian, and $\cd(G|Q')$ contains powers of $2$ that are divisible by $2^a$. Hence $a=1$, $\cd(G|Q')=\{2\}$, and $\cd(G)=\{1,2,3\}$ which is not the case. Thus in this case $G$ is not of type $5$, so it is of type four.

If $B(G)$ is the last graph in Figure~\ref{fig: 6}, then $\Gamma(G)$ has no isolated vertices and hence~\cite[Theorem 5.2]{ML2} implies that $G$ has a non-abelian normal Sylow subgroup. Now Remark~\ref{rem:20} verifies that $G$ is either a group of type $1$ or $6$. The case of Lewis type $6$ is excluded by Lemma~\ref{lemma:new1}.

 Finally, if $B(G)$ is the second graph in Figure~\ref{fig: 5}, then $G$ is either a group of type $1$, $4$, $5$ or $6$. The case of Lewis type $6$ is excluded by Lemma~\ref{lemma:new1}.

We have summarized  this remark in Table~\ref{tab:TTCR}.
}
\end{remark}

\begin{table}[ht]
\caption{Lewis types when $B(G)$ is a union of paths}
\centering
\begin{tabular}{|c|c|l|}\hline
Graph & Types & Examples \\
\hline
nr~1 Figure~~\ref{fig: 4}&1,4&$\mathtt{SmallGroup}(24,3)$ has type $1$\\
&&$\mathtt{PrimitiveSolvablePermGroup}(2,2,1,1)$ has type $4$\\
\hline
nr~2 Figure~~\ref{fig: 4}&1,2,5&$\mathtt{SmallGroup}(288,860)$ has type $1$\\
&&$\mathtt{PrimitiveGroup}(9,6)$ has type $2$\\
&&$\mathtt{SmallGroup}(48,28)$ has type $5$\\
\hline
nr~3 Figure~~\ref{fig: 4}&1&A family of examples are constructed in~\cite[Theorem 2.3]{LL}, the smallest arises by taking\\
&&(using the notation in~\cite[Theorem~$2.3$]{LL}) $p=3$, $a=2$ and $b=4$\\
\hline
nr~1 Figure~~\ref{fig: 5}&1,4&The group $PH$ in ~\cite[Example 3.4]{N} where $|\pi(k)|= 2$ is of Lewis type $1$\\
&&$\mathtt{PrimitiveSolvablePermGroup}(4,2,2,2)$ has type $4$\\
\hline
nr~2 Figure~~\ref{fig: 5}&1,4,5&See Example~\ref{exa:newnewnew} for a group of Lewis type $1$\\
&&$\mathtt{PrimitiveSolvablePermGroup}(4,2,1,4)$ has type $4$\\
&&See Example~\ref{exa:newnew} for a group of Lewis type $5$\\
\hline
nr~3 Figure~~\ref{fig: 5}&1,6&See Example~\ref{exa:new} for a group of Lewis type $1$\\
&&$\mathtt{SmallGroup}(1344,816)$ has type $6$\\
\hline
nr~1 Figure~~\ref{fig: 6}&1&A family of examples of type $1$ are constructed in~\cite[Theorem 2.3]{LL}, the smallest arises\\
&&by taking (using the notation in~\cite[Theorem~$2.3$]{LL}) $p=5$, $a=12$ and $b=24$\\
\hline
nr~2 Figure~~\ref{fig: 6}&1&$\mathtt{SmallGroup}(1920,240059)$ has type $1$ and $\cd(G)^*=\{3,5,8,15\}$\\
\hline
nr~3 Figure~~\ref{fig: 6}&1& For each three distinct primes $q$, $r$, and $s$, where $q\equiv 3\pmod 4$ and $q\equiv 1 \pmod{rs}$, \\&&there exists a solvable group $G$ with $\cd(G)=\{1,r,s,rs,q^{4},q^{5}\}$, see~\cite[Section~4]{Benjamin}.\\
&&This group has cardinality $q^{12}rs$.\\
\hline
\end{tabular}
 \label{tab:TTCR}
\end{table}

\begin{figure}
\begin{tikzpicture}
  [scale=.8,auto=left,every node/.style={circle,fill=blue!20}]

 \node (m2) at (2,1)  {};
  \node (m3) at (3,1)  {};
  \node[fill=blue] (m4) at (2,3)  {};
  \node[fill=blue] (m5) at (3,3)  {};

  \foreach \from/\to in {m2/m4,m3/m5}
    \draw (\from) -- (\to);
\end{tikzpicture}
\quad\quad\quad
\begin{tikzpicture}
  [scale=.8,auto=left,every node/.style={circle,fill=blue!20}]
  \node (m2) at (2,1)  {};
  \node[fill=blue] (m3) at (3,3)  {};
  \node[fill=blue] (m4) at (1,3)  {};
  \node (m5) at (3,1)  {};
  \node[fill=blue] (m6) at (4,3)  {};

  \foreach \from/\to in {m2/m3,m2/m4,m5/m6}
    \draw (\from) -- (\to);
\end{tikzpicture}
\quad\quad\quad
\begin{tikzpicture}
  [scale=.8,auto=left,every node/.style={circle,fill=blue!20}]
  \node (m2) at (2,1)  {};
  \node[fill=blue] (m3) at (3,3)  {};
  \node[fill=blue] (m4) at (1,3)  {};
  \node (m5) at (5,1)  {};
  \node[fill=blue] (m6) at (4,3)  {};
  \node[fill=blue] (m7) at (6,3)  {};

  \foreach \from/\to in {m2/m3,m2/m4,m5/m6,m5/m7}
    \draw (\from) -- (\to);
\end{tikzpicture}
\caption{$B(G)$ as a union of paths with $|\rho(G)|=2$}
\label{fig: 4}
\end{figure}
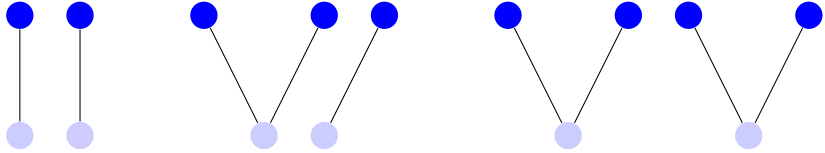

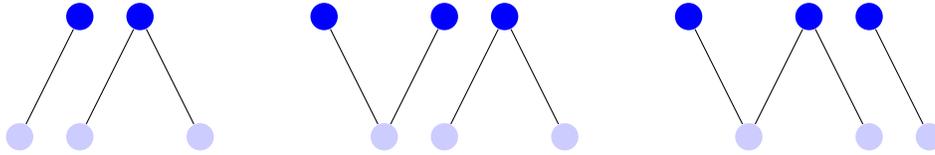
\begin{figure}
\begin{tikzpicture}
  [scale=.8,auto=left,every node/.style={circle,fill=blue!20}]
  \node (m2) at (2,1)  {};
  \node (m3) at (3,1)  {};
  \node[fill=blue] (m4) at (3,3)  {};
  \node[fill=blue] (m5) at (4,3)  {};
  \node (m6) at (5,1)  {};

  \foreach \from/\to in {m2/m4,m3/m5,m5/m6}
    \draw (\from) -- (\to);
\end{tikzpicture}
\quad\quad\quad
\begin{tikzpicture}
  [scale=.8,auto=left,every node/.style={circle,fill=blue!20}]
  \node (m2) at (2,1)  {};
  \node[fill=blue] (m3) at (1,3)  {};
  \node[fill=blue] (m4) at (3,3)  {};
  \node (m6) at (3,1)  {};
  \node[fill=blue] (m7) at (4,3)  {};
  \node (m8) at (5,1)  {};

  \foreach \from/\to in {m2/m3,m2/m4,m6/m7,m7/m8}
    \draw (\from) -- (\to);
\end{tikzpicture}
\quad\quad\quad
\begin{tikzpicture}
  [scale=.8,auto=left,every node/.style={circle,fill=blue!20}]
  \node (m2) at (8,1)  {};
  \node[fill=blue] (m3) at (7,3)  {};
  \node[fill=blue] (m4) at (9,3)  {};
  \node (m5) at (10,1)  {};
  \node (m6) at (11,1)  {};
  \node[fill=blue] (m7) at (10,3)  {};

  \foreach \from/\to in {m2/m3,m2/m4,m5/m4,m6/m7}
    \draw (\from) -- (\to);
\end{tikzpicture}
\caption{$B(G)$ is a union of paths with $|\rho(G)|=3$, and $|\cd(G)^{*}|\leq 3$}
\label{fig: 5}
\end{figure}

\begin{figure}
  \begin{tikzpicture}
  [scale=.8,auto=left,every node/.style={circle,fill=blue!20}]
  \node (m2) at (8,1)  {};
  \node[fill=blue] (m3) at (7,3)  {};
  \node[fill=blue] (m4) at (9,3)  {};
  \node[fill=blue] (m5) at (4,3)  {};
  \node (m6) at (5,1)  {};
  \node[fill=blue] (m7) at (6,3)  {};
  \node (m8) at (7,1)  {};

  \foreach \from/\to in {m2/m3,m2/m4,m6/m5,m6/m7,m7/m8}
    \draw (\from) -- (\to);
\end{tikzpicture}
\quad\quad\quad
\begin{tikzpicture}
  [scale=.8,auto=left,every node/.style={circle,fill=blue!20}]
  \node[fill=blue] (m2) at (6,3)  {};
  \node[fill=blue] (m3) at (7,3)  {};
  \node[fill=blue] (m4) at (9,3)  {};
  \node[fill=blue] (m5) at (11,3)  {};
  \node (m6) at (7,1)  {};
  \node (m7) at (8,1)  {};
  \node (m8) at (10,1)  {};

  \foreach \from/\to in {m2/m6,m7/m4,m7/m3,m8/m4,m5/m8}
    \draw (\from) -- (\to);
\end{tikzpicture}
\quad\quad\quad
\begin{tikzpicture}
  [scale=.8,auto=left,every node/.style={circle,fill=blue!20}]
  \node[fill=blue] (m2) at (6,3)  {};
    \node (m6) at (7,1)  {};
     \node[fill=blue] (m9) at (8,3)  {};
  \node[fill=blue] (m4) at (9,3)  {};
  \node[fill=blue] (m5) at (11,3)  {};
  \node (m7) at (10,1)  {};
  \node (m8) at (12,1)  {};
  \node[fill=blue] (m10) at (13,3) {};

  \foreach \from/\to in {m2/m6,m6/m9,m7/m4,m8/m5,m5/m7,m8/m10}
    \draw (\from) -- (\to);
\end{tikzpicture}
  \caption{$B(G)$ is a union of paths with $|\rho(G)|=3$, and $|\cd(G)^{*}|\geq 4$}
  \label{fig: 6}
\end{figure}
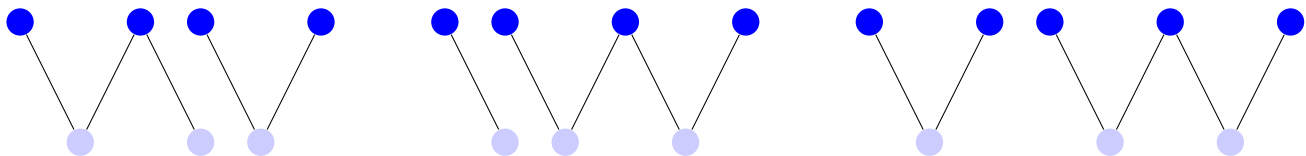

In analyzing the graphs in this section, the reader should observe how the investigation of $B(G)$
 requires the techniques developed for studying the graphs $\Gamma(G)$ and $\Delta(G)$ and {\em also} some number theoretic (or arithmetic) considerations.

We conclude this section proposing the following problem, which generalizes Question~1 in~\cite{H}. (Recall that we have solved~\cite[Question~1]{H} in Examples~\ref{ex:5} and~\ref{ex:6}.)
\begin{problem}
{\rm Determine all graphs $X$ with no cycles such that there exists a group $G$ with $X\cong B(G)$.}
\end{problem}

\subsection{Cycles and complete bipartite graphs}
Recall (from the introductory section) that Taeri~\cite{T} has proved that the bipartite divisor graph for the set of conjugacy class sizes of a finite group $G$ is a cycle if and only if it is a cycle of length six; moreover this happens if and only if $G\cong  A\times \mathrm{SL}_2(q)$, for some abelian group $A$ and some $q\in \{4,8\}$. The situation is very different and much more rich for irreducible character degrees. From ~\cite[Section 6]{LM} we know that, for every pair of odd primes $p$ and $q$ such that $p$ is congruent to $1$ modulo $3$ and $q$ is a  divisor of $p+1$, there exists a solvable group $G$ such that $\cd(G)=\{1,3q,p^{2}q,3p^{3}\}$.
This gives an example of a solvable group $G$ with $B(G)$ a cycle of length $6$.
On the other hand, among groups of order $588$, there are exactly two groups  $G$  with $B(G)$ a  cycle of length four. These groups have $\cd(G)=\{1,6,12\}$.

In ~\cite{H}, it is shown that, if $G$ is a finite group with $B(G)$  a cycle of length $n\geq 6$, then $\Delta(G)$ and $\Gamma(G)$ are cycles. This fact yields the following theorem.
\begin{theorem}[{{See~\cite[Theorem~4.5]{H}}}]~\label{thm:55}
Let $G$ be a finite group with $B(G)$ a cycle of length $n$. Then $n\in\{4,6\}$, $G$ is solvable, and $dl(G)\leq |\cd(G)|\leq 4$. In particular, if $B(G)$ is a cycle of length $4$, then there exists a normal abelian Hall subgroup $N$ of $G$ such that $\cd(G)=\{[G:I_{G}(\lambda)] : \lambda\in \Irr(N)\}$.
\end{theorem}
Since the cycle of length four is also the complete bipartite graph $K_{2,2}$, it seems natural to discuss here also the case $B(G)\cong  K_{m,n}$, for some positive integers $m\geq 2$ and $n\geq 2$. When $B(G)$ is complete bipartite, the graphs $\Delta(G)$ and $\Gamma(G)$ are both complete. Therefore, by ~\cite[Theorem 7.3]{L} or~\cite[Main Theorem]{BCLP}, we deduce that $G$ is solvable. The best  structural result on $G$ is given by Moosavi~\cite{mus}.
\begin{theorem}[{{See~\cite{mus}}}]
Let $G$ be a finite group with $B(G)$ complete bipartite. Then $G=AH$, where $A$ is an abelian normal Hall subgroup of $G$ and $H$ is either abelian or a non-abelian $p$-group for some prime $p$.
\end{theorem}

 Here we observe that there exist groups $G$ where $B(G)$ is an arbitrary complete bipartite graph. The analogous problem for the bipartite divisor graph for the set of  conjugacy class sizes seems  considerably harder; it is widely open and it is stated in~\cite{HSpiga}.
\begin{proposition}\label{prop:2^m-1}For every positive integers $m$ and $k$, there exists a group $G$ with $B(G)\cong K_{m,k}$.
\end{proposition}
\begin{proof}
Let $m$ be a positive integer and let $p_1,\ldots,p_m$ be $m$ distinct prime numbers. Set $n:=p_1\cdots p_m$. From Dirichlet's theorem on primes in arithmetic progression, there exists a prime $p$ with $p\equiv 1\pmod n$. Let $p$ be one of these primes and let $P$ be a cyclic group of order $p$. Next, let $\alpha$ be an automorphism of $P$ of order $n$ and set $H:=\langle P,\alpha\rangle$. Clearly, $H$ is a Frobenius group of order $np$, with cyclic Frobenius complement $\langle \alpha\rangle$, with cyclic Frobenius kernel $P$ and with $\cd(H)=\{1,n\}$.

Let $k$ be a positive integer and let $G:=H^k$ be the Cartesian product of $k$ copies of $H$. Clearly, $$\cd(G)=\{1,n,n^2,\ldots,n^k\}$$
and hence $B(G)$  is the complete bipartite graph $K_{m,k}$.
\end{proof}

\section{Regular Bipartite Divisor Graph}\label{sec:reg}
A graph is said to be $k$-regular if each of its vertices has valency $k$.
Since cycles are $2$-regular connected graphs, the investigation of groups $G$ with $B(G)$ a cycle has inspired the investigation~\cite{Hregular} of groups $G$ where $B(G)$ is $k$-regular. It is
clear that $0$-regular graphs (that is, empty graphs)  play no role in the study of bipartite divisor graphs. So we  start by discussing the influence
of $1$-regularity of $B(G)$ (that is, $B(G)$ is a complete matching) on the group structure of $G$. (Theorem~\ref{thm: reg} is a refinement of~\cite[Theorem~$2.1$]{Hregular},  where we have improved its statement by taking into account~\cite{BH}.)

\begin{theorem}[{{See~\cite{Hregular}}}]~\label{thm: reg}
Let $G$ be a finite group with $B(G)$ $1$-regular. Then one of the following  occurs:
\begin{itemize}
\item[(1)]$G$ is non-solvable,  $B(G)=K_2+K_2+K_2$, $G\cong  A\times \mathrm{PSL}_2(2^{n})$, where $A$ is abelian and $n\in\{2,3\}$;
\item[(2)]$G$ is solvable and one of the following cases holds:
\begin{itemize}
\item[(i)]$B(G)\cong K_2$ and $\cd(G)=\{1,p^\alpha\}$, for some prime $p$ and some positive integer $\alpha$. Moreover, either
\begin{itemize}
\item[(a)]$G\cong  P\times A$, where $P$ is a non-abelian $p$-group and $A$ is abelian, or
\item[(b)]$\alpha=1$, $\F G$ is abelian and $|G:\F G|=p$, or
\item[(c)]$G'\cap \Z G=1$ and $G/\Z G$ is a Frobenius group with kernel $(G'\times \Z G)/\Z G$ and cyclic complement of order $p^\alpha=|G:G'\times \Z G|$.
\end{itemize}
    \item[(ii)]$B(G)\cong K_2+K_2$, $h(G)\in\{2,3\}$ and $G$, with respect to its Fitting height, has one of the two structures mentioned in~\cite[Lemma 4.1]{L1998}. In particular:
        \begin{itemize}
        \item[(a)] If $h(G)=3$, then $\cd(G)=\{1,[G:{\bf F}_{2}(G)],[{\bf F}_{2}(G):\F G]\}$, where $[G:{\bf F}_{2}(G)]$ is a prime $s$ and ${\bf F}_{2}(G)/\F G$ is a cyclic $t$-group for some  prime $t\neq s$. Moreover, $G$ has Lewis type $4$.
        \item[(b)] If $h(G)=2$, then $\cd(G)=\{[G:\F G]\}\cup \cd(\F G)$, where $G/\F G$ is a cyclic $t$-group for some prime $t$ and $|\cd(\F G)|=2$. Moreover, $G$ has Lewis type $1$.
        \end{itemize}
\end{itemize}
\end{itemize}
\end{theorem}
\begin{proof}
Except for the fact that the groups in (2iia) are of Lewis type $4$ and the groups in (2iib) are of Lewis type $1$, the result follows immediately from~\cite[Theorem~$2.1$]{Hregular}  and using the main result of~\cite{BH} (when $n(B(G))=1$).

Suppose than that $G$ satisfies $B(G)=K_2+K_2$, $h(G)=3$, $\cd(G)=\{1,[G:{\bf F}_{2}(G)],[{\bf F}_{2}(G):\F G]\}$, where $[G:{\bf F}_{2}(G)]$ is a prime $s$ and ${\bf F}_{2}(G)/\F G$ is a cyclic $t$-group for some  prime $t\neq s$.  If follows readily from the description of the Lewis types and Remark~\ref{rem:20} that $G$ has type $4$ or $5$. Suppose that $G$ has type $5$. (We use the notation in~\cite[Lemma~$3.5$]{ML2}.) From~\cite[Lemma~$3.5$~(iii)]{ML2}, we deduce $2,2^a+1\in\cd(G)$. Moreover, from~\cite[Lemma~$3.5$~(iv)]{ML2}, we deduce that either $\cd(G|Q')=\emptyset$ or $\cd(G|Q') $ contains powers of $2$ that are divisible by $2^a$. Assume first that $\cd(G|Q')=\emptyset$. This means that every irreducible character of $G$ contains $Q'$ in its kernel, but this is clearly a contradiction because $Q'\ne 1$. Assume now that $\cd(G|Q')\ne\emptyset$. As $|\rho(G)|=2$, we must have $\rho(G)=\{2,2^a+1\}$ and hence $\cd(G|Q')=\{2\}$ and $a=1$. Therefore, $\cd(G)=\{1,2,3\}$. At this point to conclude we invoke~\cite[Theorem~$3.5$]{N}, which classifies the groups $X$ with $\cd(X)=\{1,m,n\}$ and $\gcd(m,n)=1$. Since $|G:\F G|=2\cdot 3=6$, we deduce that part~(1) of~\cite[Theorem~$3.5$]{N} holds. We infer that $\F G$ is abelian and hence so is $Q$, but this contradicts the description of the groups of type $5$.

Finally suppose that $G$ satisfies $B(G)=K_2+K_2$ and $h(G)=2$. Then $G$ is of Lewis type $1$ by Remark~\ref{rem:20}.
\end{proof}

The groups described in (1) and in (2i) are clear and, for each of these cases, there exists a group $G$ with $B(G)$ a complete matching.  Now, $\mathtt{SmallGroup}(320,1012)$ provides an example satisfying (2iib). The groups in~(2iia) must be of type $4$ in Lewis' sense and examples  occur plentiful ($\mathrm{Sym}(4)$ has type $4$ and $B(\mathrm{Sym}(4))=K_2+K_2$).

Let $G$ be a finite group with $B(G)$ a connected $2$-regular graph. As a connected $2$-regular graph is a cycle, by Theorem~\ref{thm:55}, $G$ is solvable with $dl(G)\leq 4$ and $B(G)$ is  a cycle of length four or six. The following theorem shows that  $B(G)$ cannot be a disconnected $2$-regular graph.

\begin{theorem}[{{See~\cite[Theorem~$3.2$ and Corollary~$3.3$]{Hregular}}}]\label{thm: 51}
Suppose that $G$ is a group with $B(G)$  $2$-regular. Then $G$ is solvable, $B(G)$ is connected and $B(G)$ is a cycle of length four or six. In particular, if $\diam(B(G))=2$, then there exists a normal abelian Hall subgroup $N$ of $G$ such that $\cd(G)=\{[G:I_{G}(\lambda)] : \lambda\in \mathrm{Irr}(N)\}$.
\end{theorem}
Theorem ~\ref{thm: 51} verifies that the union of two cycles is not the bipartite divisor graph of any finite group.

Finally, in the following two theorems, we consider the case where $B(G)$ is $3$-regular.
\begin{theorem}[{{See~\cite[Theorems~$3.4$ and~$3.5$]{Hregular}}}]~\label{thm: 4}
Let $G$ be a group with $B(G)$  $3$-regular. Then $B(G)$ is connected. Moreover, if $\Delta(G)$ is $n$-regular for $n\in\{2,3\}$, then $G$ is solvable and $\Delta(G)\cong  K_{n+1}\cong \Gamma(G)$.
\end{theorem}
\begin{theorem}[{{See~\cite{Hregular}}}]~\label{cor: 1}
Let $G$ be a solvable group with $B(G)$  $3$-regular. Then:
\begin{itemize}
\item[(i)] If at least one of $\Delta(G)$ or $\Gamma(G)$ is not complete, then $\Delta(G)$ is neither $2$-regular, nor $3$-regular.
\item[(ii)] If $\Delta(G)$ is regular, then it is a complete graph. Furthermore, if $\Gamma(G)$ is not complete, then $\Delta(G)$ is isomorphic with $K_{n}$, for $n\geq 5$.
    \end{itemize}
\end{theorem}

As a complete bipartite divisor graph $K_{m,m}$ is an $m$-regular graph, Proposition~\ref{prop:2^m-1} applied with  $m=k$ yields infinitely many solvable groups whose bipartite divisor graph is $K_{m,m}$. In particular, we obtain an example of a group whose bipartite divisor graph is a $3$-regular graph.

In Table~\ref{tab:Tcr} we give some examples of $n$-regular bipartite divisor graphs for $n\in\{1,2,3\}$.

\begin{table}[ht]
\caption{Examples of $n$-regular $B(G)$}
\centering
\begin{tabular}{|c|c|c|}
\hline & connected & disconnected \\
\hline $1$-regular & $\mathrm{Sym}(3)$ & $\mathrm{PSL}_2(8)$ \\ \hline
$2$-regular& $\mathtt{SmallGroup}(588,41)$ & Does Not Exist \\ \hline
$3$-regular & $G$ as in Proposition~\ref{prop:2^m-1} with $m=k=3$ & Does Not Exist \\
\hline
\end{tabular}
 \label{tab:Tcr}
\end{table}

As the reader can see, we know very little on groups $G$ with $B(G)$ a regular graph and on the possible bipartite regular graphs that might arise.

\begin{problem}
{\rm Construct (if possible) groups $G$ with $B(G)$ an $n$-regular graph with $n\ge 3$ and with $B(G)\ncong K_{n,n}$.}
\end{problem}

\section{Bounded order bipartite divisor graph of a finite group}\label{sec:bounded}

One of the questions that has been largely discussed by different authors is the classification of graphs that can occur as $\Delta(G)$, for some finite group $G$. To build confidence into this problem researchers have first considered graphs of bounded order. The first family of graphs that cannot  occur as $\Delta(G)$ was discovered in~\cite{BL}; later, this family was generalized in ~\cite{BJL}. These families contain graphs with arbitrarily many vertices, however they provide a great help for the problem of classifying the graphs that do occur as $\Delta(G)$, when $\Delta(G)$  has at most six vertices. For instance, in~\cite{BJLL,L3}, the authors undertake a systematic investigation on the prime degree graphs of solvable groups with six vertices and they  classify the disconnected graphs with six vertices.

Following these footsteps, in this section we study bipartite divisor graphs having at most $6$ vertices. When $B(G)$ has only two vertices,  $B(G)=K_2$ and the group $G$ has only two character degrees and a great deal is known on these groups, see~\cite{BH} and the references therein (see also Theorem~\ref{thm: reg}~(2i)). When $B(G)$ has three vertices, the classification of $G$  boils down to the understanding of groups having only two character degrees, or of groups with $\cd(G)=\{1,p^\alpha,p^\beta\}$ (which in turn is a problem on $p$-groups).

\begin{theorem}[{{See~\cite{moo4}}}]
Let $G$ be a finite group with $B(G)$ connected and having at  most four vertices. Then $G$ is solvable, $B(G)$ is one of the graphs in Figure~$\ref{fig: 11}$, and we have the following properties:
\begin{itemize}
\item[(i)] if $B(G)$ has two vertices, then $G'$ is abelian, $G=AP$, where $P\in Syl_{p}(G)$ and $A$ is an abelian normal $p$-complement;
    \item[(ii)] if $B(G)$ has three vertices, then
    \begin{itemize}
    \item[(a)]$G=AP$, where $P\in Syl_{p}(G)$ and $A$ is an abelian normal $p$-complement, or
    \item[(b)]$G'$ is abelian, $G'\cap \Z G=1$ and $\frac{G}{\Z G}$ is a Frobenius group with cyclic complement;
    \end{itemize}
    \item[(iii)] if $B(G)$ has four vertices, then
    \begin{itemize}
    \item[(c)] $G=AH$ is the semidirect product of an abelian normal subgroup $A$ and a
Hall subgroup $H$ which is either a Sylow $p$-subgroup of $G$ or an abelian $\{p,q\}$-subgroup, or
\item[(d)]$G'$ is abelian, $G'\cap \Z G=1$ and $\frac{G}{\Z G}$ is a Frobenius group with cyclic complement.
    \end{itemize}
\end{itemize}
\end{theorem}
This theorem shows that when $B(G)$ has at most four vertices the structure of the graph $B(G)$ and {\em also } the structure of the group $G$ is well-understood. (If $B(G)$ is disconnected, then $B(G)=K_2+K_2$ and this case was dealt with in the previous section.) The same behavior occurs when $B(G)$ has five vertices.

\begin{theorem}[{See~\cite{moo5}}]
Let $G$ be a finite group with $B(G)$ connected and having five vertices. Then $B(G)$ is one of the graphs in Figure~$\ref{fig: 8}$,~$\ref{fig: 9}$ or~$\ref{fig: 10}$. Furthermore, we have the following properties:
\begin{itemize}
\item[(i)] If $|\rho(G)|=1$, then $G=AP$, where $P$ is a Sylow $p$-subgroup for some prime $p$ and $A$ is a normal abelian $p$-complement.
        \item[(ii)] If $|\rho(G)|=2$, then $G$ is solvable and $G=HN$, where $H$ is either a Sylow $p$-subgroup or a Hall $\{p,q\}$-subgroup of $G$ and $N$ is a normal complement.
            \item[(iii)] If $|\rho(G)|=3$, then $G$ is solvable and one of the following cases occurs:
            \begin{itemize}
            \item[(a)] $G=HN$ where $H$ is a Sylow $p$-subgroup or a Hall $\{p,q\}$-subgroup or a Hall abelian $\{p,q,r\}$-subgroup of $G$ and $N$ is its normal complement.
                \item[(b)] $G=QN$, where $Q$ is an abelian Sylow $q$-subgroup of $G$ and $N$ is its normal complement.
                    \end{itemize}
                    \item[(iv)] If $|\rho(G)|=4$, then $G'$ is abelian, $G'\cap \Z G=1$ and $\frac{G}{\Z G}$ is a Frobenius group with cyclic complement.
\end{itemize}
\end{theorem}
The following example will be a useful tool to construct most of the groups in Table~\ref{tab:BOB}.
\begin{example}~\label{exam: mi}{\rm
Let $1<m_1<m_2<\cdots<m_r$ be distinct positive integers such that $m_i$ divides $m_{i+1}$ for all $i\in\{1,\ldots,r-1\}$. Then by ~\cite[Theorem 4.1]{N} there exists a group $G$ such that $\cd(G)=\{1,m_1,\ldots,m_r\}$. We denote this group by $G_{m_1,\ldots,m_r}$. Indeed, our Proposition~\ref{prop:2^m-1} is a very special case of this general result.}
\end{example}

\vskip 0.3 true cm
\begin{table}
\caption{Examples of $G$ with $o(B(G))\leq 5$, where it is connected}
\centering
\begin{tabular}{|c|c|c|l|}\hline
Graph & Examples & Graph & Examples\\
\hline
nr~1 Figure~~\ref{fig: 8}&$Q_{8}\times Q_{8}\times Q_{8}\times Q_{8}$&nr~1 Figure~~\ref{fig: 11}&$\mathtt{Sym}(3)$\\
\hline
nr~2 Figure~~\ref{fig: 8}&$G_{210}$ as in Example~\ref{exam: mi}&nr~2 Figure~~\ref{fig: 11}&$\mathtt{SmallGroup}(32,6)$\\
&&& with normal Sylow $2$-subgroup\\
\hline
nr~1 Figure~~\ref{fig: 9}&$\mathtt{SmallGroup}(108,17)$&nr~2 Figure~~\ref{fig: 11}&$\mathtt{SmallGroup}(96,13)$\\
&&& with non-normal Sylow $2$-subgroup\\
\hline
nr~2 Figure~~\ref{fig: 9}&$G_{2,6,12}$ as in Example~\ref{exam: mi}&nr~3 Figure~~\ref{fig: 11}&$\mathtt{SmallGroup}(42,1)$\\
\hline
nr~3 Figure~~\ref{fig: 9}&$G_{6,12,24}$ as in Example~\ref{exam: mi}&nr~4 Figure~~\ref{fig: 11}&$\mathtt{SmallGroup}(930,1)$\\
\hline
nr~4 Figure~~\ref{fig: 9}&$\mathtt{SmallGroup}(72,15)$&nr~5 Figure~~\ref{fig: 11}&$\mathtt{SmallGroup}(384,20)$\\
&&& with non-normal Sylow $2$-subgroup\\
\hline
nr~1 Figure~~\ref{fig: 10}&$G_{5,30}$ as in Example~\ref{exam: mi}&nr~5 Figure~~\ref{fig: 11}&$\mathtt{SmallGroup}(128,71)$\\
&&& with normal Sylow $2$-subgroup\\
\hline
nr~2 Figure~~\ref{fig: 10}&$G_{15,30}$ as in Example~\ref{exam: mi}&nr~6 Figure~~\ref{fig: 11}&$\mathtt{SmallGroup}(588,38)$\\
\hline
nr~3 Figure~~\ref{fig: 10}&$G_{30,60}$ as in Example~\ref{exam: mi}&nr~7 Figure~~\ref{fig: 11}&$\mathtt{SmallGroup}(96,70)$\\
\hline
nr~4 Figure~~\ref{fig: 10}&$\mathtt{SmallGroup}(960,5748)$& \\
\hline
\end{tabular}
 \label{tab:BOB}
\end{table}

\begin{figure}
\begin{tikzpicture}
[scale=.8,auto=left,every node/.style={circle,fill=blue!20}]
\node (m2) at (1,1)  {};
  \node[fill=blue] (m3) at (1,3)  {};

  \foreach \from/\to in {m2/m3}
    \draw (\from) -- (\to);
\end{tikzpicture}
\quad\quad\quad
\begin{tikzpicture}
[scale=.8,auto=left,every node/.style={circle,fill=blue!20}]
\node (m2) at (2,1)  {};
  \node[fill=blue] (m3) at (1,3)  {};
  \node[fill=blue] (m4) at (3,3)  {};

  \foreach \from/\to in {m2/m3,m2/m4}
    \draw (\from) -- (\to);
\end{tikzpicture}
\quad\quad\quad
\begin{tikzpicture}
[scale=.8,auto=left,every node/.style={circle,fill=blue!20}]
\node (m2) at (1,1)  {};
  \node (m3) at (3,1)  {};
  \node[fill=blue] (m4) at (2,3)  {};

  \foreach \from/\to in {m2/m4,m3/m4}
    \draw (\from) -- (\to);
\end{tikzpicture}
\quad\quad\quad
\begin{tikzpicture}
  [scale=.8,auto=left,every node/.style={circle,fill=blue!20}]
  \node (m2) at (1,1)  {};
  \node (m3) at (2,1)  {};
  \node (m4) at (3,1)  {};
  \node[fill=blue] (m6) at (2,3)  {};

  \foreach \from/\to in {m2/m6,m3/m6,m4/m6}
    \draw (\from) -- (\to);
\end{tikzpicture}
\quad\quad\quad
\begin{tikzpicture}
  [scale=.8,auto=left,every node/.style={circle,fill=blue!20}]
  \node (m2) at (2,1)  {};
  \node[fill=blue] (m3) at (1,3)  {};
  \node[fill=blue] (m4) at (2,3)  {};
  \node[fill=blue] (m5) at (3,3)  {};

  \foreach \from/\to in {m2/m3,m2/m4,m5/m2}
    \draw (\from) -- (\to);
\end{tikzpicture}
\quad\quad\quad
\begin{tikzpicture}
  [scale=.8,auto=left,every node/.style={circle,fill=blue!20}]
  \node (m2) at (1,1)  {};
  \node (m3) at (2,1)  {};
\node[fill=blue] (m4) at (1,3) {};
  \node[fill=blue] (m6) at (2,3)  {};

  \foreach \from/\to in {m6/m3,m6/m2,m3/m4,m2/m4}
    \draw (\from) -- (\to);
\end{tikzpicture}
\quad\quad\quad
\begin{tikzpicture}
  [scale=.8,auto=left,every node/.style={circle,fill=blue!20}]
  \node (m2) at (1,1)  {};
  \node (m3) at (2,1)  {};
\node[fill=blue] (m4) at (1,3) {};
  \node[fill=blue] (m6) at (2,3)  {};

  \foreach \from/\to in {m6/m3,m3/m4,m2/m4}
    \draw (\from) -- (\to);
\end{tikzpicture}
\caption{Connected bipartite divisor graphs of order at most four}
\label{fig: 11}
\end{figure}
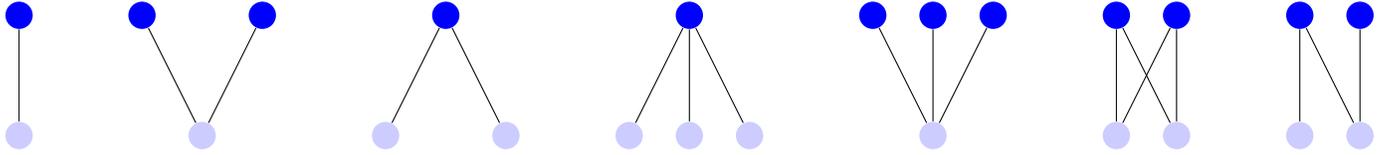

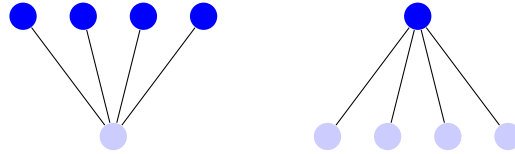
\begin{figure}
\begin{tikzpicture}
  [scale=.8,auto=left,every node/.style={circle,fill=blue!20}]
  \node (m2) at (2.5,1)  {};
  \node[fill=blue] (m3) at (1,3)  {};
  \node[fill=blue] (m4) at (2,3)  {};
  \node[fill=blue] (m5) at (3,3)  {};
  \node[fill=blue] (m6) at (4,3)  {};

  \foreach \from/\to in {m2/m3,m2/m4,m5/m2,m2/m6}
    \draw (\from) -- (\to);
\end{tikzpicture}
\quad\quad\quad
\begin{tikzpicture}
  [scale=.8,auto=left,every node/.style={circle,fill=blue!20}]
  \node (m2) at (1,1)  {};
  \node (m3) at (2,1)  {};
  \node (m4) at (3,1)  {};
  \node (m5) at (4,1)  {};
  \node[fill=blue] (m6) at (2.5,3)  {};

  \foreach \from/\to in {m6/m3,m6/m2,m6/m4,m6/m5}
    \draw (\from) -- (\to);
\end{tikzpicture}
\caption{Connected bipartite divisor graphs of order five with one or four primes}
\label{fig: 8}
\end{figure}

\begin{figure}
\begin{tikzpicture}
  [scale=.8,auto=left,every node/.style={circle,fill=blue!20}]
  \node (m2) at (2,1)  {};
  \node[fill=blue] (m3) at (1,3)  {};
  \node (m4) at (3,1)  {};
  \node[fill=blue] (m5) at (2,3)  {};
  \node[fill=blue] (m7) at (3,3)  {};

  \foreach \from/\to in {m2/m3,m2/m5,m2/m7,m4/m7}
    \draw (\from) -- (\to);
\end{tikzpicture}
\quad\quad\quad
\begin{tikzpicture}
  [scale=.8,auto=left,every node/.style={circle,fill=blue!20}]
  \node (m2) at (2,1)  {};
  \node[fill=blue] (m3) at (1,3)  {};
  \node (m4) at (3,1)  {};
  \node[fill=blue] (m5) at (2,3)  {};
  \node[fill=blue] (m7) at (3,3)  {};

  \foreach \from/\to in {m2/m3,m2/m5,m2/m7,m4/m7,m4/m5}
    \draw (\from) -- (\to);
\end{tikzpicture}
\quad\quad\quad
\begin{tikzpicture}
  [scale=.8,auto=left,every node/.style={circle,fill=blue!20}]
  \node (m2) at (2.5,1)  {};
  \node[fill=blue] (m3) at (2,3)  {};
  \node[fill=blue] (m4) at (3,3)  {};
  \node (m6) at (3.5,1)  {};
  \node[fill=blue] (m7) at (4,3)  {};

  \foreach \from/\to in {m2/m3,m2/m4,m2/m7,m7/m6,m6/m3,m6/m4}
    \draw (\from) -- (\to);
\end{tikzpicture}
\quad\quad\quad
\begin{tikzpicture}
  [scale=.8,auto=left,every node/.style={circle,fill=blue!20}]
  \node (m2) at (1.5,1)  {};
  \node (m3) at (2.5,1)  {};
  \node[fill=blue] (m5) at (1,3)  {};
  \node[fill=blue] (m6) at (2,3)  {};
  \node[fill=blue] (m7) at (3,3)  {};

  \foreach \from/\to in {m2/m5,m2/m6,m6/m3,m3/m7}
    \draw (\from) -- (\to);
\end{tikzpicture}
\caption{Connected bipartite divisor graphs of order five and two primes}
\label{fig: 9}
\end{figure}
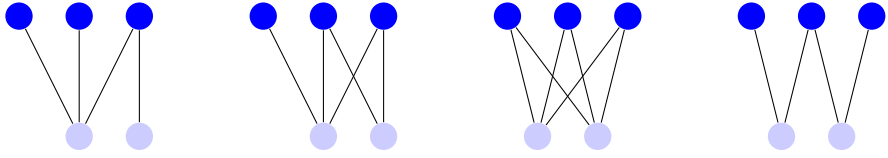

\begin{figure}
\begin{tikzpicture}
  [scale=.8,auto=left,every node/.style={circle,fill=blue!20}]
  \node (m2) at (1,1)  {};
  \node (m3) at (2,1)  {};
  \node (m4) at (3,1)  {};
  \node[fill=blue] (m5) at (2,3)  {};
  \node[fill=blue] (m7) at (3,3)  {};

  \foreach \from/\to in {m2/m5,m3/m5,m5/m4,m4/m7}
    \draw (\from) -- (\to);
\end{tikzpicture}
\quad\quad\quad
\begin{tikzpicture}
  [scale=.8,auto=left,every node/.style={circle,fill=blue!20}]
  \node (m2) at (1,1)  {};
  \node (m3) at (2,1)  {};
  \node (m4) at (3,1)  {};
  \node[fill=blue] (m5) at (2,3)  {};
  \node[fill=blue] (m7) at (3,3)  {};

  \foreach \from/\to in {m2/m5,m3/m5,m5/m4,m4/m7,m3/m7}
    \draw (\from) -- (\to);
\end{tikzpicture}
\quad\quad\quad
\begin{tikzpicture}
  [scale=.8,auto=left,every node/.style={circle,fill=blue!20}]
   \node (m2) at (1,1)  {};
  \node (m3) at (2,1)  {};
  \node (m4) at (3,1)  {};
  \node[fill=blue] (m5) at (1.5,3)  {};
  \node[fill=blue] (m7) at (2.5,3)  {};

  \foreach \from/\to in {m2/m5,m3/m5,m5/m4,m4/m7,m3/m7,m7/m2}
    \draw (\from) -- (\to);
\end{tikzpicture}
\quad\quad\quad
\begin{tikzpicture}
  [scale=.8,auto=left,every node/.style={circle,fill=blue!20}]
  \node (m2) at (1,1)  {};
  \node (m3) at (2,1)  {};
  \node (m4) at (3,1)  {};
  \node[fill=blue] (m5) at (1.5,3)  {};
  \node[fill=blue] (m7) at (2.5,3)  {};

  \foreach \from/\to in {m2/m5,m3/m5,m4/m7,m3/m7}
    \draw (\from) -- (\to);
\end{tikzpicture}
\caption{Connected bipartite divisor graphs of order five and three primes}
\label{fig: 10}
\end{figure}
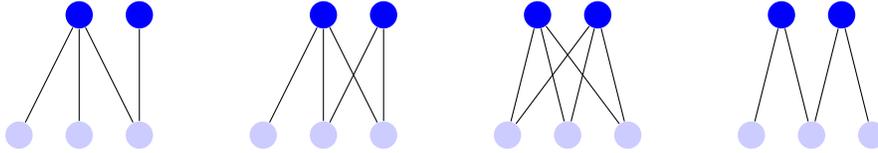

\begin{remark}{\rm
Let $G$ be a finite group with $B(G)$ disconnected and having five vertices. A case-by-case analysis yields  that $B(G)$ is one of the graphs in Figure~\ref{fig: 1}. In particular, $B(G)$ is a union of two paths and hence the structure of $G$ is described in Section~\ref{unionpaths}.
}
\end{remark}

\begin{figure}
\begin{tikzpicture}
  [scale=.8,auto=left,every node/.style={circle,fill=blue!20}]
  \node (m2) at (2,1)  {};
  \node[fill=blue] (m3) at (1,3)  {};
  \node[fill=blue] (m4) at (3,3)  {};
  \node (m5) at (3,1)  {};
  \node[fill=blue] (m6) at (4,3)  {};

  \foreach \from/\to in {m2/m3,m2/m4,m5/m6}
    \draw (\from) -- (\to);
\end{tikzpicture}
\quad\quad\quad
\begin{tikzpicture}
  [scale=.8,auto=left,every node/.style={circle,fill=blue!20}]
  \node (m2) at (1,1)  {};
  \node (m3) at (3,1)  {};
  \node (m4) at (4,1)  {};
  \node[fill=blue] (m5) at (3,3)  {};
  \node[fill=blue] (m6) at (2,3)  {};

  \foreach \from/\to in {m6/m3,m6/m2,m5/m4}
    \draw (\from) -- (\to);
\end{tikzpicture}
\caption{Disconnected bipartite divisor graphs of order five}
\label{fig: 1}
\end{figure}
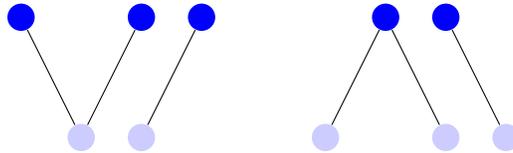

\begin{figure}
\begin{tikzpicture}
  [scale=.8,auto=left,every node/.style={circle,fill=blue!20}]
  \node (m2) at (2,1)  {};
  \node[fill=blue] (m3) at (1,3)  {};
  \node[fill=blue] (m4) at (3,3)  {};
  \node (m5) at (3,1)  {};
  \node[fill=blue] (m6) at (2,3)  {};
  \node[fill=blue] (m7) at (4,3)  {};

  \foreach \from/\to in {m2/m3,m2/m4,m2/m6,m5/m7}
    \draw (\from) -- (\to);
\end{tikzpicture}
\quad\quad\quad
\begin{tikzpicture}
  [scale=.8,auto=left,every node/.style={circle,fill=blue!20}]
  \node (m2) at (5,1)  {};
  \node[fill=blue] (m3) at (3,3)  {};
  \node (m4) at (2,1)  {};
  \node[fill=blue] (m5) at (1,3)  {};
  \node[fill=blue] (m6) at (4,3)  {};
  \node[fill=blue] (m7) at (6,3)  {};

  \foreach \from/\to in {m4/m3,m2/m6,m5/m4,m2/m7}
    \draw (\from) -- (\to);
\end{tikzpicture}
\quad\quad\quad
\begin{tikzpicture}
  [scale=.8,auto=left,every node/.style={circle,fill=blue!20}]
  \node (m2) at (1,1)  {};
  \node (m3) at (2,1)  {};
  \node (m4) at (3,1)  {};
  \node[fill=blue] (m5) at (2,3)  {};
  \node (m6) at (4,1)  {};
  \node[fill=blue] (m7) at (3,3)  {};

  \foreach \from/\to in {m2/m5,m3/m5,m4/m5,m6/m7}
    \draw (\from) -- (\to);
\end{tikzpicture}
\caption{Disconnected bipartite divisor graphs of order six (part 1)}
\label{fig: 2}
\end{figure}
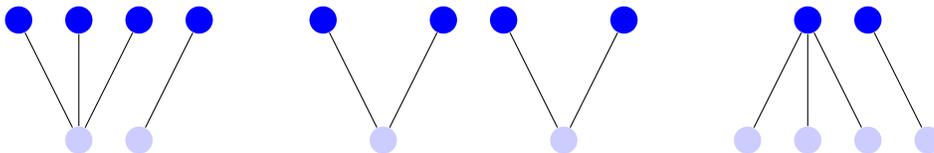

\begin{figure}
\begin{tikzpicture}
  [scale=.8,auto=left,every node/.style={circle,fill=blue!20}]
  \node (m2) at (1,1)  {};
  \node[fill=blue] (m3) at (1,3)  {};
  \node (m4) at (2,1)  {};
  \node[fill=blue] (m5) at (2,3)  {};
  \node (m6) at (3,1)  {};
  \node[fill=blue] (m7) at (3,3)  {};

  \foreach \from/\to in {m2/m3,m4/m5,m5/m6,m6/m7}
    \draw (\from) -- (\to);
\end{tikzpicture}
\quad\quad\quad
\begin{tikzpicture}
  [scale=.8,auto=left,every node/.style={circle,fill=blue!20}]
  \node (m2) at (1,1)  {};
  \node[fill=blue] (m3) at (1,3)  {};
  \node (m4) at (2,1)  {};
  \node[fill=blue] (m5) at (2,3)  {};
  \node (m6) at (3,1)  {};
  \node[fill=blue] (m7) at (3,3)  {};

  \foreach \from/\to in {m2/m3,m4/m5,m6/m7,m5/m6,m4/m7}
    \draw (\from) -- (\to);
\end{tikzpicture}
\quad\quad\quad
\begin{tikzpicture}
  [scale=.8,auto=left,every node/.style={circle,fill=blue!20}]
  \node (m2) at (2,1)  {};
  \node[fill=blue] (m3) at (1,3)  {};
  \node[fill=blue] (m4) at (3,3)  {};
  \node (m6) at (3,1)  {};
  \node[fill=blue] (m7) at (4,3)  {};
  \node (m8) at (5,1)  {};

  \foreach \from/\to in {m2/m3,m2/m4,m6/m7,m7/m8}
    \draw (\from) -- (\to);
\end{tikzpicture}
\quad\quad\quad
\begin{tikzpicture}
  [scale=.8,auto=left,every node/.style={circle,fill=blue!20}]
  \node (m2) at (1,1)  {};
  \node (m3) at (2,1)  {};
  \node (m4) at (3,1)  {};
  \node[fill=blue] (m5) at (1,3)  {};
  \node[fill=blue] (m6) at (2,3)  {};
  \node[fill=blue] (m7) at (3,3)  {};

  \foreach \from/\to in {m2/m5,m3/m6,m4/m7}
    \draw (\from) -- (\to);
\end{tikzpicture}
\caption{Disconnected bipartite divisor graphs of order six and three primes (part 2)}
\label{fig: 3}
\end{figure}
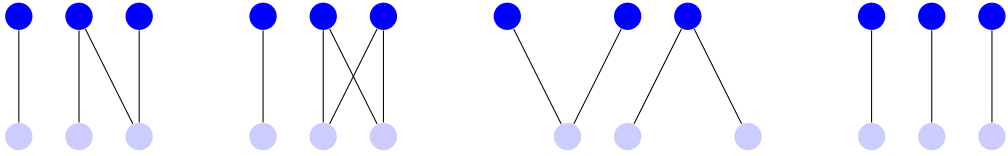
\vskip 0.4 true cm


Finally in the last theorem of this paper, we look at disconnected bipartite graphs with six vertices, and we attempt to
determine whether each graph can or cannot occur as the bipartite divisor graph of a solvable group.
\begin{theorem}~\label{thm: six}
Let $G$ be a finite group with $B(G)$  disconnected  and with six vertices. Then $B(G)$ is one of the graphs in Figure~$\ref{fig: 2}$ or~$\ref{fig: 3}$. Furthermore we have the following properties:
\begin{itemize}
\item[(i)] $n(B(G))=3$ if and only if $G\cong  A\times \mathrm{PSL}_2(2^{n})$, where $n\in\{2,3\}$ and $A$ is an abelian group.
\item[(ii)] If $|\rho(G)|=4$, then $B(G)$ is the third graph in Figure~$\ref{fig: 2}$, $G$ is solvable and it is of Lewis type one or four with the structure explained in~\cite[Lemma 4.1]{L1998}.
\item[(iii)]If $|\rho(G)|=2$, then $G$ is solvable and $B(G)$ is one the first two graphs in Figure~$\ref{fig: 2}$. If $B(G)$ is the second graph in Figure~$\ref{fig: 2}$, then $G$ is a group of Lewis type one. If $B(G)$ is the first graph in Figure~$\ref{fig: 2}$, then $G$ is of Lewis type one or five. 
\item[(iv)] If $|\rho(G)|=3$ and $n(B(G))=2$, then $B(G)$ is one of the first three graphs in Figure~$\ref{fig: 3}$. If $G$ is non-solvable, then one of the following cases holds:
    \begin{itemize}
 \item[(a)] $G$ has a normal subgroup $U$ such that $U\cong  \mathrm{PSL}_{2}(q)$ or $\mathrm{SL}_{2}(q)$ for some odd $q\geq 5$ and if $C=\cent G U$, then $C\leq \Z G$ and $G/C\cong  \mathrm{PGL}_{2}(q)$; or
     \item[(b)] $G$ has a normal subgroup of index $2$ that is a direct product of $\mathrm{PSL}_{2}(9)$ and a central subgroup $C$. Furthermore, $G/C\cong  M_{10}$.
 \end{itemize}
 If $G$ is solvable, then it is of Lewis type one or six when $B(G)$ is one of the first two graphs in Figure~$\ref{fig: 3}$ and of Lewis type one, four, or five when $B(G)$ is the third graph in Figure~$\ref{fig: 3}$.
\end{itemize}
\end{theorem}
\begin{proof}
As $n(B(G))>1$ and $B(G)$ has six vertices, we have $|\rho(G)|+|\cd(G)^{*}|=6$, where $|\rho(G)|\in\{2,3,4\}$. So we consider three cases with respect to $|\rho(G)|$.

If $|\rho(G)|=4$, then $|\cd(G)|=3$. Now ~\cite[Theorem 12.15]{IS} and ~\cite[Corollary 4.2]{L} imply that $G$ is solvable of derived length at most $3$ and each connected component of $\Delta(G)$ is a complete graph. On the other hand, ~\cite[Theorem 4.3]{L} verifies that $K_{2}+K_{2}$ is not the prime degree graph of a solvable group, so we conclude that $B(G)$ is the third graph in Figure ~\ref{fig: 2}. Now by Remark~\ref{rem:20} we conclude that $G$ is of Lewis types one, four, five, or six. Similar to the proof of Lemma~\ref{lemma:new1}, we can see that $G$ is not a group of Lewis type six. If $G$ is of Lewis type five, then considering the notations in ~\cite[Lemma 3.5]{ML2} we conclude that either $2^{a}=2$ and $\cd(G)=\{1,2,3\}$ which contradicts the structure of $B(G)$ or $\cd(G|Q^{'})=\emptyset$ which is not possible as $Q$ is non-abelian. Therefore, $G$ is of Lewis type one or four.

If $|\rho(G)|=2$, then $B(G)$ has two connected components. It is easy to see that $B(G)$ is one of the first two graphs in Figure ~\ref{fig: 2}. First suppose that $B(G)$ is the second graph in Figure ~\ref{fig: 2}. By ~\cite[Theorem 7.1]{L}, the case $\Gamma(G)\cong K_{2}+K_{2}$ is impossible for a non-solvable group, therefore $G$ is solvable and, by Table~\ref{tab:TTCR}, $G$ is of Lewis type one. Now consider the first graph in Figure ~\ref{fig: 2}. As every character degree of $G$ is a power of some prime and $|\rho(G)|=2$, by ~\cite[Theorem 30.3]{Hu1}, we conclude that $G$ is solvable with $\Delta(G)\cong K_{1}+K_{1}$ and, by Remark~\ref{rem:20}, we deduce that $G$ is of Lewis type one, four, or five. Similar to the proof of Lemma~\ref{lemma:new}, we see that the case where $G$ has Lewis type four does not occur. Hence $G$ is either of Lewis type one or five.

Finally consider the case where $|\rho(G)|=|\cd(G)^{*}|=3$. If $n(B(G))=3$, then $B(G)$ is $1$-regular and, by Theorem~\ref{thm: reg} we can observe that $G\cong  A\times \mathrm{PSL}_{2}(2^{n})$, where $n\in\{2,3\}$ and $A$ is an abelian group. Assume that $n(B(G))=2$ and $G$ is non-solvable. Since $|\cd(G)|=4$, ~\cite[Theorem A]{GA} implies that $G$ has one of the following structures:
 \begin{itemize}
 \item[(a)] $G$ has a normal subgroup $U$ such that $U\cong  \mathrm{PSL}_{2}(q)$ or $\mathrm{SL}_{2}(q)$ for some odd $q\geq 5$ and if $C=\cent G U$, then $C\leq \Z G$ and $G/C\cong  \mathrm{PGL}_2(q)$; or
     \item[(b)] $G$ has a normal subgroup of index $2$ that is a direct product of $\mathrm{PSL}_{2}(9)$ and a central subgroup $C$. Furthermore, $G/C\cong  M_{10}$.
 \end{itemize}
  In particular by ~\cite[Corollary B]{GA}, $\cd(G)=\{1,q-1,q,q+1\}$ for some odd prime power $q>3$ or $\cd(G)=\{1,9,10,16\}$. Consequently, $\Delta(G)$ has an isolated vertex and $B(G)$ is one of the first two graphs in Figure~\ref{fig: 3}. When $G$ is solvable, by Remark~\ref{rem:20} we can see that $G$ is a group of Lewis type one, four, five, or six. Suppose $G$ is a group of Lewis type four or five. As $G$ has no non-abelian normal Sylow subgroup, ~\cite[Theorem 5.2]{ML2} implies that one connected component of $\Gamma(G)$ contains only one degree $a$ where the prime divisors of $a$ lie in the larger connected component of $\Delta(G)$. None of the first two graphs in Figure~\ref{fig: 3} satisfies this property, so in these cases $G$ is not a group of Lewis type four or five. Hence it is of Lewis type one or six.

 If $B(G)$ is the third graph in Figure~\ref{fig: 3}, then, by Table~\ref{tab:TTCR}, it is a group of Lewis type one, four, or five.
\end{proof}
\begin{example}{\rm
Here we give some examples of solvable groups whose bipartite divisor graphs have six vertices, are disconnected, but are not a union of paths.
  \begin{itemize}
    \item For $G=\mathtt{SmallGroup}(320,1581)$, we have $\cd(G)=\{1,2,4,8,5\}$, $B(G)$ is the first graph in Figure~\ref{fig: 2}, and $G$ is of Lewis type five.

    \item Let $p$ be an  odd prime and let $a$ be a positive integer with $|\pi(p^{a}-1)|\geq 3$ (the smallest case is $p=31$ and $a=1$). Let $b$ be a divisor of $p^{a}-1$ with exactly three prime divisors. Let $E$ be the extraspecial group of order $p^{2a+1}$ with exponent $p$. Then $E$ has an automorphism $\varphi$ of order $b$ which centralizes $\Z E$. Let $G:=E\rtimes \langle\varphi\rangle$. Then $\cd(G)=\{1,b,p^{a}\}$ and $B(G)$ is the third graph in Figure~\ref{fig: 2}, and $G$ is of Lewis type one.
  \end{itemize}}
\end{example}

We were not able to find a group $G$ with $B(G)$ isomorphic to the second graph in Figure~\ref{fig: 3}. We leave this as an open question.
\begin{question}\label{question:11111}
Is there a finite group $G$ with $B(G)$ isomorphic to the second graph in Figure $\ref{fig: 3}$?
\end{question}


\end{document}